\documentclass[12pt,a4paper,reqno]{amsart}

\usepackage{amsmath,amssymb,bbm,comment,graphicx,enumitem,mathrsfs,psfrag,ragged2e,relsize,setspace,stmaryrd,tikz,trimclip,upgreek,xcolor}

\usetikzlibrary{calc,decorations.pathmorphing,shapes}

\usepackage[bookmarksnumbered,colorlinks,plainpages,pdfstartview=FitBH,citecolor=blue,linkcolor=red]{hyperref}

\definecolor{refkey}{rgb}{1,.25,.25}
\definecolor{labelkey}{rgb}{0.25,1,.75}

\definecolor{citecolor}{rgb}{1,.25,.25}

\setlist[itemize]{labelindent=6pt,labelwidth=\parindent,leftmargin=!}

\newlist{subquestion}{enumerate}{1}
\setlist[subquestion,1]{label=$\circ$}

\makeatletter
\def\SL@inlinetext#1{%
  \SL@interlinetextright{\SL@prlabelname{#1}}%
}
\def\SL@interlinetextleft{\SL@setlefttrue\SL@interlinetext}
\def\SL@interlinetextright{\SL@setleftfalse\SL@interlinetext}
\def\SL@interlinetext#1{%
  \setbox\@tempboxa=\hbox{\showlabelsetlabel{\SL@prlabelname{#1}}}\dp\@tempboxa\z@
  \ifvmode
    \nointerlineskip\vbox to 0pt{
      \hbox to \columnwidth{\box\@tempboxa}}
  \else
    \ifSL@setleft
      \hbox to 0pt{%
        \hss
        \vbox to 0pt{\vss
          \hbox to 0pt{\hss\box\@tempboxa}%
          \showlabelrefline
        }}%
    \else
      \hbox to 0pt{%
        \vbox to 0pt{\vss
          \box\@tempboxa
          \showlabelrefline
        }\hss}%
    \fi
    \penalty10000
  \fi
}
\makeatother

\DeclareMathAlphabet{\matheu}{U}{eur}{m}{n}
\DeclareSymbolFont{euleroperators}{U}{eur}{m}{n}
\SetSymbolFont{euleroperators}{bold}{U}{eur}{b}{n}
\makeatletter
\renewcommand{\operator@font}{\mathgroup\symeuleroperators}
\makeatother

\makeatletter 
\DeclareRobustCommand{\cev}[1]{%
  \mathpalette\do@cev{#1}%
}
\newcommand{\do@cev}[2]{%
  \fix@cev{#1}{+}%
  \reflectbox{$\m@th#1\vec{\reflectbox{$\fix@cev{#1}{-}\m@th#1#2\fix@cev{#1}{+}$}}$}%
  \fix@cev{#1}{-}%
}
\newcommand{\fix@cev}[2]{%
  \ifx#1\displaystyle
    \mkern#23mu
  \else
    \ifx#1\textstyle
      \mkern#23mu
    \else
      \ifx#1\scriptstyle
        \mkern#22mu
      \else
        \mkern#22mu
      \fi
    \fi
  \fi
}

\makeatother

\makeatletter
\def\@seccntformat#1{%
  \protect\textup{\protect\@secnumfont
    \ifnum\pdfstrcmp{subsection}{#1}=0 \bfseries\fi
    \csname the#1\endcsname
    \protect\@secnumpunct
  }%
}
\makeatother

\makeatletter
\def\@cite#1#2{[\textbf{#1}\if@tempswa , #2\fi]}
\def\@biblabel#1{[\textbf{#1}]}
\makeatother

\numberwithin{equation}{section}
\numberwithin{figure}{section}

\renewcommand{\thefigure}{\arabic{section}.\arabic{figure}}
\renewcommand{\thefigure}{\ifnum\value{section}>0 \arabic{section}.\fi\arabic{figure}}
\renewcommand{\theequation}{\ifnum\value{section}>-1 \arabic{section}.\fi\arabic{equation}}

\newtheorem{thm}[equation]{Theorem}

\newtheorem{cor}[equation]{Corollary}

\newtheorem{prp}[equation]{Proposition}
\newtheorem{tme}[equation]{Theorem-Example}

\theoremstyle{definition}

\theoremstyle{remark}
\newtheorem{rem}[equation]{Remark}

\newcommand{\thmref}[1]{Theorem~\ref{#1}}
\newcommand{\prpref}[1]{Proposition~\ref{#1}}

\newcommand{\corref}[1]{Corollary~\ref{#1}}
\newcommand{\tmeref}[1]{Theorem-Example~\ref{#1}}

\newcommand{\remref}[1]{Remark~\ref{#1}}

\newcommand{\figref}[1]{Figure~\ref{#1}}
\newcommand{\secref}[1]{Section~\ref{#1}}

\makeatletter
\@namedef{subjclassname@2020}{%
  \textup{2020} Mathematics Subject Classification}
\makeatother

\newcounter{sarrow}

\newlength{\mqheight}
\newlength{\mqnormalheight}
\newcommand{\mathquote}[1]{%
    \settoheight{\mqheight}{\hbox{\ensuremath{#1}}}%
    \addtolength{\mqheight}{-\mqnormalheight}%
    \text{\raisebox{\the\mqheight}{``}}%
    #1%
    \text{\raisebox{\the\mqheight}{''}}%
}

\newcommand\1{\mathbbm{1}}

\newcommand\Ac{\mathcal A}

\newcommand{\af}{\mathfrak a}
\newcommand\al{\alpha}

\newcommand\Bc{\mathcal B}
\newcommand\be{\beta}

\renewcommand{\bf}{\mathfrak b}
\newcommand{\bw}{\mathbin{\text{\scriptsize $\bigcurlywedge$}_{a}^{\!\!\ga}}}

\newcommand\cb{\mathbin{\ooalign{$\circledcirc$\cr\hidewidth$\mathrel{\mkern-0.1mu\raisebox{-0.12ex}{$\mathlarger{\mathlarger{\bullet}}$}}$\hidewidth}}}
\newcommand\Cc{\mathcal C}

\newcommand\de{\delta}
\newcommand\deu{\updelta}
\newcommand{\diag}{\operatorname{diag}}
\newcommand{\diam}{\operatorname{diam}}

\newcommand{\ep}{\varepsilon}

\newcommand{\G}{\mathcal G}

\newcommand\ga{\gamma}

\newcommand{\Geom}{\operatorname{Geom}}
\newcommand\GL{\operatorname{GL}}

\newcommand{\Hb}{\mathbf H}
\newcommand\Hc{\mathcal H}

\newcommand{\Ib}{\mathbf I}

\newcommand{\Iso}{\operatorname{Iso}}

\newcommand{\ka}{\varkappa}
\newcommand\kab{\boldsymbol{\kappa}}

\newcommand\la{\lambda}
\newcommand\lab{\boldsymbol{\lambda}}

\newcommand{\mapstoto}{\mathop{\,\sim\joinrel\rightsquigarrow\,}}

\newcommand\ol{\overline}

\newcommand\om{\omega}

\newcommand\p{\partial}

\newcommand\PSL{\operatorname{PSL}}

\newcommand\QQ{\mathbb Q}

\newcommand\Rc{\mathcal R}
\newcommand\RR{\mathbb R}

\newcommand\Sc{\mathcal S}

\newcommand{\si}{\sigma}
\DeclareMathOperator{\sgn}{sgn}
\newcommand\SL{\operatorname{SL}}
\newcommand{\sm}{\,\setminus\,}
\newcommand\sst{\scriptscriptstyle}
\newcommand\supp{\operatorname{supp}}

\newcommand\Tc{\mathcal T}

\renewcommand\th{\theta}
\newcommand\thb{\boldsymbol{\theta}}

\newcommand\Uc{\mathcal U}

\newcommand\ZZ{\mathbb Z}

\makeatletter
\DeclareRobustCommand{\st}{%
  \mathrel{\mathpalette\short@to\relax}%
}

\newcommand{\short@to}[2]{%
  \mkern2mu
  \clipbox{{.3\width} 0 0 0}{$\m@th#1\vphantom{+}{\pmb\shortrightarrow}$}%
  }
\makeatother

\makeatletter
\patchcmd{\@addmarginpar}{\ifodd\c@page}{\ifodd\c@page\@tempcnta\m@ne}{}{}
\makeatother

\AtBeginDocument{\def\MR#1{}} 

\begin{document}

\title[Singularity of compound stationary measures]{Singularity of compound stationary measures}


\author{Behrang Forghani}
\email{forghanib@cofc.edu}

\author{Vadim A.\ Kaimanovich}
\email{vadim.kaimanovich@gmail.com}


\subjclass[2020]{}

\thanks{The authors were supported by NSF grant DMS 2246727 and by NSERC grant RGPIN-2022-05066, respectively}

\dedicatory{To the memory of Anatoly Moiseevich Vershik}

\begin{abstract}
We show that the product or convex combination of two Markov operators with equivalent stationary measures need not have a stationary measure from the same measure class. More specifically, we exhibit examples of a hitherto undescribed phenomenon: maximal entropy random walks for which the resulting compound random walks no longer have maximal entropy. The underlying group in these examples is $\PSL(2,\ZZ)\cong\ZZ_2*\ZZ_3$, and the associated harmonic measures belong to the canonical Minkowski and Denjoy measure classes on the boundary. These examples also demonstrate that a number of other natural families of random walks are not closed under convolutions or convex combinations of step distributions.
\end{abstract}

\maketitle

\thispagestyle{empty}


\section*{Introduction}

\textbf{\S 1. Deterministic vs Markov dynamics.} The concept of singularity dates back to the period of rapid development of rigorous mathematical analysis during the second half of the 19th century, when, to quote \textsc{Poincaré} \cite{Poincare99}, \emph{on vit surgir toute une foule de fonctions bizarres qui semblaient s'efforcer de ressembler aussi peu que possible aux honnêtes fonctions qui servent à quelque chose.} Measure theory that was born at the beginning of the 20th century rendered this notion symmetrical and much more ``honest'' than Poincaré had imagined: instead of singularity of a function (with respect to the Lebesgue measure which had not yet been defined at the time), one could now talk about mutual singularity of two different measures. The emergence of ergodic theory in the 1930s further advocated a ``pluralistic universe'' in which diverse measures happily coexist, and their mutual singularity is a common phenomenon, as, for instance, is the case for any pair of ergodic invariant measures of the same transformation.

The notion of a measure preserving transformation (or, equivalently, of an invariant measure) is at the very heart of ergodic theory. Obviously, if two transformations preserve the same measure, then their composition also does, which is why one can further talk about semigroups and groups of measure preserving transformations.

Koopman's operator approach to dynamics consists in passing from a transformation~$T$ of a state space~$X$ to the linear operator $K_T f = f\circ T$ acting on an appropriate space of functions~$f$ on $X$, and the dual of the Koopman operator is the result of the functorial extension of the map $T$ from $X$ to the space of measures on $X$. From this perspective, it is natural to take into account the linear structure on the operator space, or, if one wants to keep preserving positivity and total mass, the convex structure. This consideration leads to the notion of a \textsf{Markov operator}. It is determined by a map \textsf{(Markov kernel)} $x\mapsto \uppi_x$ from the state space~$X$ to the space of probability measures on it (in the deterministic case $\uppi_x$ is the delta-measure~$\de_{Tx}$). We follow the probabilistic tradition by using the same notation $P$ both for the Markov operator acting on functions on the state space as $Pf(x)=\langle f, \uppi_x\rangle$ (this is precisely the Koopman operator in the deterministic case) and for the dual operator acting on measures on the state space as $\la P=\int \uppi_x \,d\la(x)$, cf.\ Dirac's bra-ket notation in quantum mechanics.

Algebraically, the space of Markov operators on a common state space (subject to appropriate continuity or measurability assumptions) is closed with respect to the compatible operations of taking products and convex combinations, i.e., is a \textsf{convex algebra}. We say that a Markov operator is a \textsf{compound} of several other operators if it belongs to the convex algebra generated by these operators, i.e., is a convex combination of their (generally non-commuting) products. In particular, given two Markov operators $P_1,P_2$, their product and their convex combinations will be collectively referred to as \textsf{simple compound operators} derived from $P_1$ and $P_2$. Clearly, any compound can be obtained by an iterative application of the operation of taking just a simple compound of appropriate pairs of operators.

In probabilistic terminology (to be used in this paper) invariant measures $\th$ of (the dual of) a Markov operator $P$, i.e., such that $\th P=\th$, are called \textsf{stationary} \emph{(below we consider probability stationary measures only).} We recall that if the state space $X$ is compact, and a Markov kernel is weakly$^*$ continuous, then (essentially) by the classical Krylov -- Bogolyubov theorem the arising Markov operator on $X$ always has a probability stationary measure.

\medskip

\textbf{\S 2. Stationary measures of compound operators.} By extending our earlier observation, we can now say that \emph{if two Markov operators~$P_1,P_2$ on the same state space have a common stationary measure, then this measure is also stationary for any compound of $P_1$ and $P_2$} (for brevity, we will often refer to stationary measures of compound operators as \textsf{compound stationary measures}). What happens when~$P_1$ and~$P_2$ have respective stationary measures~$\th_1$ and $\th_2$ which are merely equivalent, but do not coincide?

\medskip

\noindent
\textbf{Question A.} Let $P$ be a simple compound (i.e., the product or a convex combination) of two Markov operators $P_1$ and $P_2$ on a common state space. If $P_1$ and $P_2$ have distinct equivalent stationary measures, does $P$ have a stationary measure from the same measure class?

\medskip

In what concerns the product of two operators, this question makes sense even in the deterministic case, i.e., when $P_1$ and $P_2$ are the Koopman operators corresponding to two transformations $T_1,T_2$ acting on the same state space. In this situation Question A amounts to asking whether the equivalence of invariant measures of $T_1$ and $T_2$ implies the existence of an invariant measure within the same measure class for the composition of~$T_1$ and $T_2$. We are not aware of any discussion of this latter question in the literature, although it would not be surprising if the corresponding counterexamples are known to specialists.

Below we will be interested in a special class of ``space homogeneous'' Markov operators arising in the situation when the state space $X$ is endowed with an action (again, subject to appropriate continuity or measurability assumptions) of a countable group $G$. In this case just a single probability measure $\mu$ on $G$ gives rise to the associated Markov operator $P=P_\mu$ acting on measures $\th$ on $X$ as convolution $\th P_\mu=\mu*\th$. The map $\mu\mapsto P_\mu$ is an antihomomorphism of the convex algebra of probability measures on $G$ (with respect to convolution) to the convex algebra of Markov operators on $X$. Stationary measures of the operator $P_\mu$ are called \textsf{$\mu$-stationary}, see \textsc{Furstenberg -- Glasner} \cite{Furstenberg-Glasner10} and \textsc{Bowen~-- Hartman -- Tamuz} \cite{Bowen-Hartman-Tamuz17} for a general discussion of stationary dynamical systems. For simplicity, \emph{we always require the measure $\mu$ to be \textsf{non-degenerate}} (i.e., require its support to generate the group $G$ as a semigroup); then one can immediately see that \emph{any $\mu$-stationary measure is quasi-invariant with respect to the action of the group $G$.} The specialization of Question A to this situation~is

\medskip

\noindent
\textbf{Question B.} Given an action of a countable group $G$ on a space $X$, let $\mu$ be a simple compound of two probability measures $\mu_1,\mu_2$ on $G$ (i.e., either their convolution, or a convex combination). If there are distinct equivalent $\mu_1$-stationary and $\mu_2$-stationary measures on~$X$, does there exist a $\mu$-stationary measure in the same measure class?

\medskip

The above question can also be asked about any concrete quasi-invariant measure class on the considered $G$-space.

\medskip

\noindent
\textbf{Question B$'$.} Given an action of a countable group $G$ on a space $X$ endowed with a quasi-invariant measure class $\thb$, let $\mu_1,\mu_2$ be two probability measures on $G$ that have distinct respective stationary measures from the class $\thb$, and let $\mu$ be a simple compound of the measures $\mu_1,\mu_2$. Does there exist a $\mu$-stationary measure in the class~$\thb$?

\medskip

Question B$'$ is closely related to the question about \textsf{stationarizing} a given quasi-invariant measure $\th$ (i.e., finding a measure $\mu$ on the group such that $\mu*\th=\th$), or, in a weaker form, stationarizing a given quasi-invariant measure class~$\thb$ (i.e., finding a measure $\mu$ on the group and a measure $\th$ from the class $\thb$ with $\mu*\th=\th$).\footnotemark\ In full generality this is a highly non-trivial problem, cf.\ \textsc{Elliott -- Giordano} \cite{Elliott-Giordano93}, \textsc{Glasner -- Weiss}~\cite{Glasner-Weiss16}, and the references therein; however, it was satisfactorily settled in the hyperbolic context, see \S\textbf{3} below.

\footnotetext{\;General constructions that allow one, starting from a probability measure~$\mu_1$ on a group $G$, to obtain other measures $\mu_2$ on the same group such that any $\mu_1$-stationary measure is also $\mu_2$-stationary or such that any $\mu_1$-stationary measure is equivalent to a $\mu_2$-stationary measure are described by the authors in \cite{Forghani-Kaimanovich15p}.}

\medskip

\textbf{\S 3. Boundary action of hyperbolic groups.} Specializing the situation further, let us additionally assume that $G$ is a \textsf{non-elementary word hyperbolic group} which naturally acts on its \textsf{hyperbolic boundary} $\p G$. It is classically known that in this setup there exists a \emph{unique} $\mu$-stationary measure~$\nu$ on $\p G$; it is the \textsf{primary harmonic measure} of the random walk on the group with the step distribution $\mu$ issued from the group identity, see \secref{sec:HarD} for details.\footnotemark\ In particular, if two step distributions $\mu_1,\mu_2$ on the group~$G$ have the same harmonic measure, then it is also the harmonic measure of any compound of~$\mu_1$ and~$\mu_2$. Note that harmonic measures are ergodic (for instance, due to the aforementioned uniqueness), and therefore the harmonic measures of two different step distributions are always either equivalent or singular.

\footnotetext{\;See \textsc{Maher -- Tiozzo} \cite{Maher-Tiozzo18,Maher-Tiozzo21} for the general case of groups of isometries of (not necessarily proper) Gromov hyperbolic spaces. Our Question C extends to this more general setup as well.}

Any left-invariant distance $\rho$ on the group $G$ quasi-isometric to a word distance gives rise to the associated ergodic quasi-invariant \textsf{conformal measure class}~$\lab_\rho$ on the hyperbolic boundary $\p G$. The class $\lab_\rho$ contains the Patterson measures of distance~$\rho$ as well as the Hausdorff measure of an appropriately defined metric on the boundary (whose Hausdorff dimension is equal to the exponential growth rate $v=v_\rho$ of distance~$\rho$), see \textsc{Coornaert} \cite{Coornaert93} and \textsc{Furman} \cite{Furman02a}.

We say that a measure $\mu$ on a hyperbolic group $G$ is \textsf{filling} with respect to a fixed left-invariant distance $\rho$ on the group if the harmonic measure $\nu$ of the random walk $(G,\mu)$ belongs to the corresponding conformal measure class $\lab_\rho$. In the words of \textsc{Vershik} \cite[p.~669]{Vershik00}, see \S\textbf{4} below, this means that the random walk ``provides asymptotically complete knowledge about elements of the group''.

\medskip

\noindent
\textbf{Question C.} Let $G$ be a hyperbolic group endowed with a left-invariant distance $\rho$ quasi-isometric to a word one, and let $\mu$ be a simple compound of two $\rho$-filling probability measures~$\mu_1,\mu_2$ on $G$ (i.e., either their convolution, or a convex combination). Is the measure~$\mu$ also $\rho$-filling?

\medskip

For the measure classes $\lab_\rho$ on the hyperbolic boundary $\p G$ (as well as for individual measures from these classes with reasonable densities) the stationarization problem mentioned in \S\textbf{2} was solved by \textsc{Connell -- Muchnik} \cite[Theorem 1.19]{Connell-Muchnik07}. However, the stationarizing measures on $G$ arising in their construction are inherently infinitely supported, which gives rise to the question whether there exists a \emph{finitely supported} measure on $G$ that stationarizes a concrete measure class $\lab_\rho$, or, in our terminology, whether there are finitely supported measures filling with respect to $\rho$. On the other hand, \textsc{Blachere~-- Ha\"{\i}ssinsky -- Mathieu} \cite{Blachere-Haissinsky-Mathieu11} showed that any symmetric finitely supported measure is filling with respect to the associated Green distance.

It is natural to ask \emph{which measure classes $\lab_\rho$ admit no filling finitely supported measures,} i.e., the harmonic measure of any finitely supported random walk on $G$ is \emph{singular} with respect to $\lab_\rho$. It was conjectured by \textsc{Kaimanovich -- Le Prince} \cite{Kaimanovich-LePrince11} that this is the case if $G$ is the fundamental group of a compact manifold with constant negative curvature and $\rho$ is the distance on $G$ induced by the Riemannian metric of the universal covering manifold;\footnotemark\ we believe that it should be true for compact negatively curved manifolds of non-constant sectional curvature as well. In spite of a recent progress achieved in certain particular cases (see \textsc{Carrasco -- Lessa -- Paquette} \cite{Carrasco-Lessa-Paquette21}, \textsc{Kosenko}~\cite{Kosenko21}, and \textsc{Kosenko~-- Tiozzo} \cite{Kosenko-Tiozzo22}), the singularity conjecture remains wide open in full generality. In a different direction \textsc{Gou\"ezel -- Math\'eus -- Maucourant} \cite[Theorem 1.5]{Gouezel-Matheus-Maucourant18} proved that $\lab_\rho$ admits no filling finitely supported measures (actually, a bit stronger, no filling measures with finite superexponential moments) in the situation when the group $G$ is not virtually free, and $\rho$ is a word metric on $G$.

\footnotetext{\;The \textsf{singularity conjecture} was actually formulated in \cite{Kaimanovich-LePrince11} for general discrete subgroups of semi-simple Lie groups.}

\medskip

\textbf{\S 4. Maximal entropy measures.} Let now $G$ be again a general countable group endowed with a left-invariant distance $\rho$ of finite exponential growth rate $v$. If $\mu$ is a probability measure on $G$ with a finite first moment $|\mu|$ with respect to $\rho$, then its entropy $H(\mu)$ is also finite, and therefore both the \textsf{rate of escape} $\ell=\ell(G,\mu)=\lim_n |\mu^{*n}|/n$ and \textsf{asymptotic entropy} $h=h(G,\mu)=\lim_n H(\mu^{*n})/n$ are well-defined and finite. Moreover, if $G$ is non-amenable (which is, for instance, the case for non-elementary word hyperbolic groups considered in \S\textbf{3}), then both these quantities are strictly positive, see \textsc{Kaimanovich~-- Vershik}~\cite{Kaimanovich-Vershik83}. The rate of escape, the asymptotic entropy, and the growth rate satisfy the inequality $h\le \ell v$ first established by \textsc{Guivarc'h} \cite[Propositions 1 and 2 on~pp.~74-75]{Guivarch80} and later called \textsf{fundamental} by \textsc{Vershik} \cite{Vershik00} (it is a direct consequence of the Shannon -- McMillan~-- Breiman type equidistribution property of the asymptotic entropy).

We say that a measure $\mu$ on $G$ has \textsf{maximal entropy} (with respect to distance $\rho$) if the fundamental inequality holds as an equality and all the involved quantities are strictly positive. A particular instance of this notion arises in the situation when both the measure~$\mu$ and the distance $\rho$ are determined by a finite symmetric generating set $K\subset G$, so that $\mu=\mu_K$ is the uniform distribution on $K$, and $\rho=\rho_K$ is the word distance determined by $K$; if the measure $\mu_K$ has maximal entropy with respect to the distance $\rho_K$, then \textsc{Vershik} \cite{Vershik00} called the generating set $K$ \textsf{extremal}.

\medskip

\noindent
\textbf{Question D.} Let $G$ be a countable group endowed with a left-invariant distance $\rho$ of finite exponential growth rate, let $\mu_1,\mu_2$ be two probability measures on $G$ that have maximal entropy with respect to $\rho$, and let $\mu$ be a simple compound of $\mu_1$ and $\mu_2$ (i.e.,~$\mu$ is either their convolution, or a convex combination). Is the measure $\mu$ also a maximal entropy one?

\medskip

This question is closely related to Question C (moreover, we are not aware of any results on maximal entropy random walks beyond the hyperbolic context), so let us return for a moment to the setup of \S\textbf{3}. Rewritten as $h/\ell\le v$, the fundamental inequality becomes an inequality between the Hausdorff dimensions of the harmonic measure and of the hyperbolic boundary $\p G$ itself. This interpretation (analogous to the relationship between entropy, Lyapunov exponents and Hausdorff dimension in smooth dynamics) was first given for Kleinian groups by \textsc{Ledrappier} \cite{Ledrappier83}.\footnotemark\ Thus, any filling measure with a finite first moment necessarily has maximal entropy.

\footnotetext{\;Ledrappier used one of several possible definitions of the Hausdorff dimension of a measure. However, in the hyperbolic setup the harmonic measure is exact dimensional just under the finite first moment condition without any further assumptions whatsoever, and therefore all reasonable definitions of its dimension coincide, see \textsc{Tanaka} \cite{Tanaka19} and \textsc{Dussaule -- Yang} \cite{Dussaule-Yang23}.}

In the opposite direction \textsc{Blach\`ere -- Ha{\"{\i}}ssinsky -- Mathieu} \cite[Theorem~1.5]{Blachere-Haissinsky-Mathieu11} proved that any finitely supported maximal entropy measure $\mu$ is filling (see also the discussion of this result by \textsc{Gou\"ezel -- Math\'eus -- Maucourant} \cite[Theorem~1.2]{Gouezel-Matheus-Maucourant18}). As pointed out in \cite{Gouezel-Matheus-Maucourant18}, in view of \textsc{Gou\"ezel}'s work \cite{Gouezel15}, this is true for the measures $\mu$ with finite superexponential moments as well. These results were extended to relatively hyperbolic groups by \textsc{Dussaule -- Gekhtman} \cite{Dussaule-Gekhtman20}.

Since the rate of escape and the asymptotic entropy of the \textsf{reflected measure} $\check\mu(g)=\mu\left(g^{-1}\right)$ coincide with those of the original measure $\mu$, if $\mu$ has maximal entropy, then~$\check\mu$ also has. Therefore, the aforementioned results from \cite{Blachere-Haissinsky-Mathieu11} and \cite{Gouezel-Matheus-Maucourant18} also imply that for any maximal entropy measure $\mu$ with finite superexponential moments its harmonic measure is equivalent to the harmonic measure of the reflected random walk.

The key property behind these results is the quasi-multiplicativity of the Green kernel of the arising random walks along geodesics on the group that goes back to \textsc{Ancona} \cite{Ancona88} in the finitely supported case. This property makes the picture similar to that of Gibbs measures in symbolic dynamics. More specifically, dealing with $G$-invariant Radon measures \textsf{(geodesic currents)} on the square $\p^2 G=\p G\times\p G \setminus\diag$ of the hyperbolic boundary allows one to talk about ``invariant measures'' of the geodesic flow on hyperbolic groups without formally defining the latter (this program was outlined by the second author \cite{Kaimanovich90, Kaimanovich94}).

The Ancona property ultimately implies that the product of the harmonic measures of the forward and the backward random walks is in this situation equivalent to a unique (up to a constant multiplier) ``Gibbs'' \textsf{harmonic geodesic current}. Then the fact that the filling measures are precisely the maximal entropy ones is analogous to the classical Parry's theorem, which established the uniqueness of maximal entropy invariant measures for subshifts of finite type.\footnotemark\ Although the work of \textsc{Blach\`ere -- Ha{\"{\i}}ssinsky -- Mathieu} does not use geodesic currents (see the discussion in \cite[p.~684]{Blachere-Haissinsky-Mathieu11}), the ``current approach'' has been further developed in the recent works of \textsc{Gekhtman -- Tiozzo} \cite{Gekhtman-Tiozzo20} and \textsc{Cantrell~-- Tanaka} \cite{Cantrell-Tanaka24}.

\footnotetext{\;The situation with geodesic currents is, actually, more subtle, as each of the underlying distances on $G$ determines its own maximal entropy current. Therefore, one should rather talk about the corresponding variational principle, cf. \textsc{Cantrell~-- Tanaka} \cite{Cantrell-Tanaka24}.}

As for the full generality of the measures $\mu$ with a finite first moment, we see no reason why maximal entropy measures would necessarily be filling, although we are not aware of any counterexamples (cf.\ the discussions of the dynamical systems admitting several invariant measures with maximal entropy by \textsc{Denker -- Grillenberg~-- Sigmund} \cite[Section 19]{Denker-Grillenberg-Sigmund76}, Haydn \cite{Haydn98p}, and the references therein). We would not be surprised to see an example in which the forward random walk is filling, while the backward (reflected) one is not, either. Moreover, we conjecture that harmonic geodesic currents do not exist for general finite first moment measures $\mu$.

\medskip

\textbf{\S 5. Singularity of Poisson boundaries.} As we have already mentioned at the beginning of \S\textbf{3}, from the point of view of stochastic processes, any probability measure~$\mu$ on a group $G$ can be considered as the step distribution of the associated random walk on the group. The \textsf{Poisson boundary} $\p_\mu G$ of this random walk (i.e., the space of ergodic components of the time shift on the path space of the random walk) is endowed with the \textsf{primary harmonic measure} $\nu$ (that describes the limit distribution of the sample paths issued from the group identity) which is $\mu$-stationary, and the action of the group $G$ on~$\p_\mu G$ is ergodic. In a sense, the measure space $(\p_\mu G,\nu)$ is the maximal one among all the $\mu$-stationary measure spaces (see \textsc{Kaimanovich -- Vershik} \cite{Kaimanovich-Vershik83}, \textsc{Kaimanovich} \cite{Kaimanovich00a}, \textsc{Furstenberg~-- Glasner} \cite{Furstenberg-Glasner10}). In particular, the triviality of the Poisson boundary $\p G_\mu$ means that any $\mu$-stationary measure is actually invariant.

Since the Poisson boundary $\p_\mu G$ is defined in terms of the path space of a specific step distribution $\mu$, \emph{a priori} there is no natural way to identify the Poisson boundaries of different measures on the same group. The only meaningful comparison arises when these boundaries can be realized on the same measure $G$-space. In this case, the corresponding primary harmonic measures may either outright coincide (in which case we talk about the \textsf{same} Poisson boundary) or be merely equivalent (in which case we talk about \textsf{equivalent} Poisson boundaries). Otherwise, if the Poisson boundaries of two different step distributions cannot be realized on the same measure space, we say that they are \textsf{singular}. Our forthcoming paper \cite{Forghani-Kaimanovich15p} describes general constructions that allow one to derive, from a given measure $\mu$, other step distributions with the same or equivalent Poisson boundaries.

\medskip

\noindent
\textbf{Question E.} Let $\mu$ be a simple compound of two probability measures $\mu_1,\mu_2$ on the same group $G$ with equivalent Poisson boundaries.
Is the Poisson boundary $\p_\mu G$ also equivalent to them?

\medskip

Of course, rather than just declaring that the Poisson boundary is the space of ergodic components of the time shift, it is natural to seek a more concrete description of the Poisson boundary in terms of the ``observed'' limit behaviour of the random walk's sample paths. In other words, one may ask whether it can be \emph{identified} with one of the ``boundary'' spaces associated with the group on the basis of its more specific analytic, geometric, algebraic, or combinatorial features. In what concerns hyperbolic groups discussed in~\S\textbf{3} (or, more generally, discrete groups of isometries of Gromov hyperbolic spaces), for any measure $\mu$ with a finite entropy (in particular, with a finite first moment) its Poisson boundary coincides with the hyperbolic boundary endowed with the unique $\mu$-stationary measure by a recent result of \textsc{Chawla -- Forghani -- Frisch -- Tiozzo} \cite{Chawla-Forghani-Frisch-Tiozzo22p} (although the identification problem for hyperbolic groups still remains open in full generality). Thus, in this situation Question E is equivalent to Question B.

However, in general Question E (unlike Question~B) is non-trivial already in the situation when the Poisson boundaries $\p_{\mu_1} G$ and $\p_{\mu_2} G$ are the same. In particular, we are not aware of any examples in which $\p_{\mu_1} G,\p_{\mu_2} G$ are both trivial, whereas $\p_\mu G$ is not (cf.\ the widely open question about the \textsf{stability of the Liouville property}, i.e., whether the Poisson boundaries of finitely supported symmetric measures on a finitely generated group are either all trivial or all non-trivial simultaneously).

\medskip

\textbf{\S 6. The main result and further questions.} Our principal result is a negative answer to Questions A-E obtained as a combination of Theorems-Examples \ref{tme:ex0}, \ref{tme:ex1}, and \ref{tme:ex2} from \secref{sec:sex}.

\medskip

\noindent
\textbf{Main Theorem.}
For the modular group $\G=\ZZ_2*\ZZ_3\cong\PSL(2,\ZZ)$ endowed with the word distance determined by the generators of $\mathbb{Z}_2$ and $\mathbb{Z}_3$, the set of filling probability measures is not closed under convolutions or convex combinations. More specifically, there are pairs of finitely supported probability measures $\mu_1$ and $\mu_2$ on $\G$ such that both $\mu_1$ and $\mu_2$ are filling, whereas their convolution $\mu_1*\mu_2$ (resp., convex combination $t\mu_1+(1-t)\mu_2$) is not filling.

\medskip

It then prompts

\medskip

\noindent
\textbf{Question F.} Under what additional conditions do the answers to questions A-E become positive?

\medskip

We strongly suspect that the phenomenon described in our Main Theorem is quite common, and is likely to occur for all hyperbolic groups. However, we
do not address the problem in this full generality (cf.\ the discussion of the singularity conjecture in \S\textbf{3}). Instead, we exhibit the corresponding measures explicitly by leveraging certain special features of the modular group.

One limitation of our examples is that the measures involved are not symmetric. While symmetry is not a natural requirement when talking about convolutions (as the convolution of two symmetric measures is by no means symmetric in general), in what concerns convex combinations it is really interesting to know whether the answer to questions~\mbox{B-E} would still be negative under the additional assumption that the step distributions are symmetric. In particular, the following question can, to a certain degree, be regarded as an extension of the problem of stability of the Liouville property (already mentioned in~\S\textbf{5}).

\medskip

\noindent
\textbf{Question G.} Let $G$ be a hyperbolic group endowed with a left-invariant distance $\rho$ quasi-isometric to a word one, and let $\mu$ be a convex combination of two $\rho$-filling \emph{symmetric} probability measures~$\mu_1,\mu_2$ on $G$. Is the measure~$\mu$ also $\rho$-filling?

\medskip

Although the set of step distributions on a given hyperbolic group $G$ with a prescribed harmonic measure is far from being fully understood (see the discussion in \S\textbf{3}), we at least know that it has a convex algebra structure. The Main Theorem shows that this is no longer the case if one talks about the set of step distributions with the harmonic measure from a prescribed quasi-invariant measure class on $\p G$.

\medskip

\noindent
\textbf{Question H.}
What can one say about the structural properties of the set of step distributions on a hyperbolic group $G$ whose harmonic measures belong to a given quasi-invariant measure class on the hyperbolic boundary $\p G$? More specifically, what can be said about the set of filling step distributions with respect to a fixed left-invariant distance on $G$?

\medskip

For answering this question one might need a classification of the left-invariant distances quasi-isometric to a word distance (or, equivalently, of the corresponding geodesic currents).

\medskip

\pagebreak

\textbf{\S 7. The construction.} As we have explained, a negative answer to Question C within the class of finitely supported measures would also imply negative answers to Questions A, B, D, and E. Therefore, our goal is twofold: first, to find a pair of filling finitely supported step distributions, $\mu_1$ and $\mu_2$, with distinct harmonic measures (because identical harmonic measures for $\mu_1$ and $\mu_2$ would result in the same harmonic measure for any compound $\mu$); and second, to demonstrate that the resulting compound step distribution $\mu$ is not filling (i.e., its harmonic measure is singular with respect to the conformal measure class).

The first candidates that come to mind when talking about random walks with a non-trivial behaviour at infinity are free groups (since \textsc{Dynkin -- Malutov} \cite{Dynkin-Malutov61}). However, they are not well-suited for our purposes: the explicitly known harmonic measures of the nearest neighbour random walks on free groups are all pairwise singular (see \textsc{Kaimanovich} \cite[Theorem~3.11]{Kaimanovich25}, although this observation essentially originates from \textsc{Levit -- Molchanov} \cite{Levit-Molchanov71}). For other finitely supported step distributions (let alone infinitely supported ones) the available descriptions of the harmonic measures are not explicit enough despite recent progress in this direction by \textsc{Bordenave~-- Dubail} \cite{Bordenave-Dubail23}.

Compared to free groups, free products of finite groups may look like awkward ``virtual'' relatives, overshadowed by the perfect symmetry of the homogeneous Cayley trees. Yet, to borrow La Rochefoucauld's famous maxim, \emph{nos vertus ne sont, le plus souvent, que des vices déguisés}.

We call \textsf{intermittent} the free products of the form $G=\ZZ_2*K$, where~$K$ is a finite group with at least 3 elements to make sure that $G$ is non-elementary (i.e., not virtually cyclic).\footnotemark\ The hyperbolic boundary~$\p G$ consists then of the infinite words formed by sequences of elements from $K'=K\setminus \{e\}$ interspersed with the generator $a$ of the group \mbox{$\ZZ_2=\{e,a\}$}. Depending on whether the first letter is $a$ or not, \emph{the boundary splits into two disjoint components, and each of them can be identified with the unilateral full shift on the alphabet}~$K'$. This setup has two advantages: first, it is much easier to work with the full shift (and Bernoulli measures on~it) than with topological Markov chains burdened with admissibility conditions; second, the presence of two copies of the same full shift naturally facilitates the construction of families of distinct equivalent measures by varying the weights of each component.

\footnotetext{\;The second author gratefully acknowledges numerous discussions with R.~I.~Grigorchuk on self-similar presentations and their impact on random walks, which inspired his interest in this class of groups as a potential source of new examples.}

For simplicity, we focus on the case $K=\ZZ_3$, in which the resulting intermittent group is the \textsf{modular group} $\G=\PSL(2,\ZZ)\cong\ZZ_2*\ZZ_3$. The infinite words from the hyperbolic boundary $\p\G$ provide the well-known \textsf{mediant encoding} of the real line (see \secref{sec:med} and \secref{sec:bdry}), with the above boundary components corresponding to the positive and the negative real rays, respectively.

The calculation of the Green kernel for nearest neighbour random walks on free groups by \textsc{Dynkin -- Malyutov} \cite{Dynkin-Malutov61} relied critically on the observation that almost every sample path, on its way to the boundary, must consecutively visit all points along the corresponding geodesic ray. The same idea was later applied to finding the Green kernel for nearest neighbour random walks on free products by \textsc{Cartwright -- Soardi} \cite{Cartwright-Soardi86} and by \textsc{Woess} \cite{Woess86a}.

Our first observation is that if, instead of simple random walks on $\G=\ZZ_2*\ZZ_3$, one passes to more general step distributions supported on the finite set $\Sc=\bigl\{a,b,b^2,ba,b^2a\bigr\}$\,---\,where~$a$ and $b$ are the respective generators of the cyclic groups $\ZZ_2$ and $\ZZ_3$\,---\,then the above technique is still applicable. This implies that for any non-degenerate step distribution $\mu$ with $\supp\mu\subset\Sc$, its harmonic measure $\nu$ is described by just two parameters $\al, p$ from the real interval $(0,1)$. Namely, $\nu$ is a convex combination $\ka^{\al,p}$ of the Bernoulli measures with the alphabet $\bigl\{b,b^2\bigr\}$ and the base distribution $(\al,1-\al)$ on the aforementioned boundary components (the infinite words that either begin or do not begin with letter~$a$), weighted by $p$ and $(1-p)$, respectively. In terms of the mediant encoding these Bernoulli measures correspond to the classical Minkowski (for $\al=\frac12$) and Denjoy (for~$\al\neq\frac12$) measures on the unit interval.

We establish an explicit description of the level subsets $\al=Const$ of the 4-dimensional simplex of probability measures on $\Sc$ as certain quadric hypersurfaces, which are easily seen to be non-convex. In particular, since it is for $\al=\frac12$ that the measures $\ka^{\frac12,p}$ belong to the conformal measure class determined by the word distance, this implies that, in this setup, a compound formed as a convex combination of filling measures need not be filling.\footnotemark

\footnotetext{\;For nearest neighbour random walks, that is, for the step distributions $\mu$ supported on the set $\bigl\{a,b,b^2\bigr\}$, the level set of the value $\al=\frac12$ is no longer quadric and instead degenerates into the convex set of symmetric measures, i.e, those satisfying $\mu(b)=\mu\bigl(b^2\bigr)$. Thus, although nearest random walks on the modular group do provide examples (for $\al\neq\frac12$) where the harmonic measure for a convex combination of step distributions is singular with respect to the common measure class of the harmonic measures of the original step distributions, such an example cannot be obtained for the conformal measure class.}

As for the compounds obtained by taking convolutions of step distributions, we rely on the general fact that the harmonic measure of the conjugate step distribution $g\mu g^{-1}$ is the translate $g\nu$ of the harmonic measure $\nu$ determined by the original step distribution $\mu$, and therefore these two harmonic measures are equivalent. We take $g$ to be the generator~$a$ of the group $\ZZ_2$, so that the convolution $\mu*a\mu a$ of $\mu$ and its conjugate $a\mu a$ is the convolution square of the translate $\mu a$, which has the same harmonic measure as $\mu a$ itself. Now, right multiplication by $a$ preserves the set $\Sc\cup\{e\}$. Thus, if we take a measure $\mu$ supported on~$\Sc$, then the harmonic measure of $\mu a$ is the same as that of the step distribution $\mu'$ also supported on $\Sc$ and obtained from $\mu a$ by removing the weight at the group identity, with a subsequent renormalization. By using our explicit description of the filling step distributions supported on $\Sc$, it is then easy to exhibit an example of a filling $\mu$ such that~$\mu a$ is not filling.

\medskip

\textbf{\S 8. Organization of the paper.} In the first part (\secref{sec:1}) we review the background and establish the necessary auxiliary results. We begin by recalling general facts about free products and their Cayley graphs (\secref{sec:cayley}), the modular group (\secref{sec:modu}), the mediant tree (\secref{sec:med}), and the resulting encoding of the real line (\secref{sec:bdry}). Although these topics are well-known, they are scattered across a number of sources with differing notation and emphases. For the sake of consistency, we provide a uniform treatment of these concepts. We have also included a brief outline of the relation of the mediant encoding with continued fractions (\secref{sec:cf}).

In \secref{sec:MD} we introduce the family of measures $\ka^{\al,p}$ on the hyperbolic boundary $\p\G$ and show that they are precisely the multiplicative Markov measures on the topological Markov chain that describes the infinite words from the hyperbolic boundary $\p\G$. After that, in \secref{sec:RN} we provide yet another characterization of these measures to be used later in the proof of \thmref{thm:u}: as unique solutions of the Radon -- Nikodym problem for product cocycles of the boundary action of the modular group (\prpref{prp:RN}). Finally, in \secref{sec:H} we identify the Hausdorff measure class on the hyperbolic boundary~$\p\G$ with the Minkowski class of the measures $\ka^{\frac12,p}$.

We begin the second part (\secref{sec:2}) with a discussion, in \secref{sec:rwmod}, of the random walks on the modular group $\G$ whose step distribution $\mu$ is supported on the set $\Sc=\bigl\{a,b,b^2,ba,b^2a\bigr\}$. Following this, in \secref{sec:HarD} we establish the first of our technical tools: \thmref{thm:ka}, which states that the harmonic measures of these random walks belong to the family $\bigl\{\ka^{\al,p}\bigr\}$. The second tool is \thmref{thm:u}, proved in \secref{sec:pass}, that provides an explicit description of the dependence of the parameters $\al$ and $p$ on the step distribution $\mu$.

As a consequence, we produce, for each value of the parameter $\al$, the quadric in the simplex of probability measures on $\Sc$ that consists of the step distributions whose harmonic measure belongs to the corresponding Denjoy class $\kab^\al$. To reduce the number of parameters, we then examine the traces of these quadrics on two 2-dimensional subsimplices: the first one consisting of the symmetric step distributions $\mu$ supported on the generating set $\Ac=\bigl\{a,b,b^2\bigr\}$ (\secref{sec:nn}), and the second one determined by the conditions $\mu(b)=\mu(ba)$ and $\mu\bigl(b^2a\bigr)=0$ (\secref{sec:ano}). Finally, in \secref{sec:sex} we consolidate all our preparatory work, and present the promised examples.

\section{Modular group, mediant tree, and Minkowski -- Denjoy measures} \label{sec:1}

\subsection{Free product and its Cayley graph} \label{sec:cayley}

We consider the \emph{free product} $\G=\ZZ_2*\ZZ_3$ of cyclic groups
\begin{equation} \label{eq:z23}
\ZZ_2=\bigl\langle\, a \, | \, a^2 \,\bigr\rangle = \{e,a\}\;, \qquad
\ZZ_3= \bigl\langle\, b \, | \, b^3 \,\bigr\rangle = \left\{ e,b,\ol b \right\} \;,
\end{equation}
where $\ol b=b^{-1}=b^2$, and put
$$
\Bc = \left\{ b,\ol b \right\} \;, \qquad \Ac= \{a\} \cup \Bc = \left\{a,b,\ol b\right\} \;.
$$
The group $\G$ consists of the words in alphabet $\Ac$ subject to the usual \textsf{admissibility (irreducibility) rule}: letter $a$ and the letters from $\Bc$ are \textsf{intermittent}, i.e., only the pairs $ab \,,\; a\ol b \,,\; ba \,,\; \ol b a$ of consecutive letters are \textsf{admissible}. The group operation in~$\G$ consists, as in any free product, in concatenation of two words with, if necessary, subsequent transformation to an irreducible form, and the group identity $e$ is represented by the empty word.

The (right) \textsf{Cayley graph} of group $\G$ with respect to generating set $\Ac$ is obtained by joining group elements with directed edges $g\xrightarrow{h} gh$ labelled with generators $h\in\Ac$. It consists of \textsf{triangles} (whose edges are labelled with order 3 generator $b$ in one direction and with its inverse~$\ol b$ in the opposite direction) joined with \textsf{bridges} labelled with order~2 generator $a$, see \figref{fig:graph}. As usual, we denote by $|g|$ the \textsf{(word) length} of a group element $g\in\G$ with respect to generating set $\Ac$, i.e., the distance from $g$ to the group identity in the Cayley graph.

\begin{figure}[h]
\begin{center}
\psfrag{a}[cl][cl]{$a$}
\psfrag{b}[cl][cl]{$b$}
\psfrag{c}[cr][cr]{$\ol b$}
\psfrag{d}[cl][cl]{$ba$}
\psfrag{e}[cr][cr]{$e$}
\psfrag{f}[cr][cr]{$\ol b a$}
\psfrag{g}[cr][cr]{$ab$}
\psfrag{h}[cl][cl]{$a\ol b$}
\psfrag{X}[cr][cr]{$\p\G|_b$}
\psfrag{Y}[cl][cl]{$\p\G|_{\ol b}$}
\psfrag{Z}[c][c]{$\p\G|_a$}
\includegraphics[scale=.6]{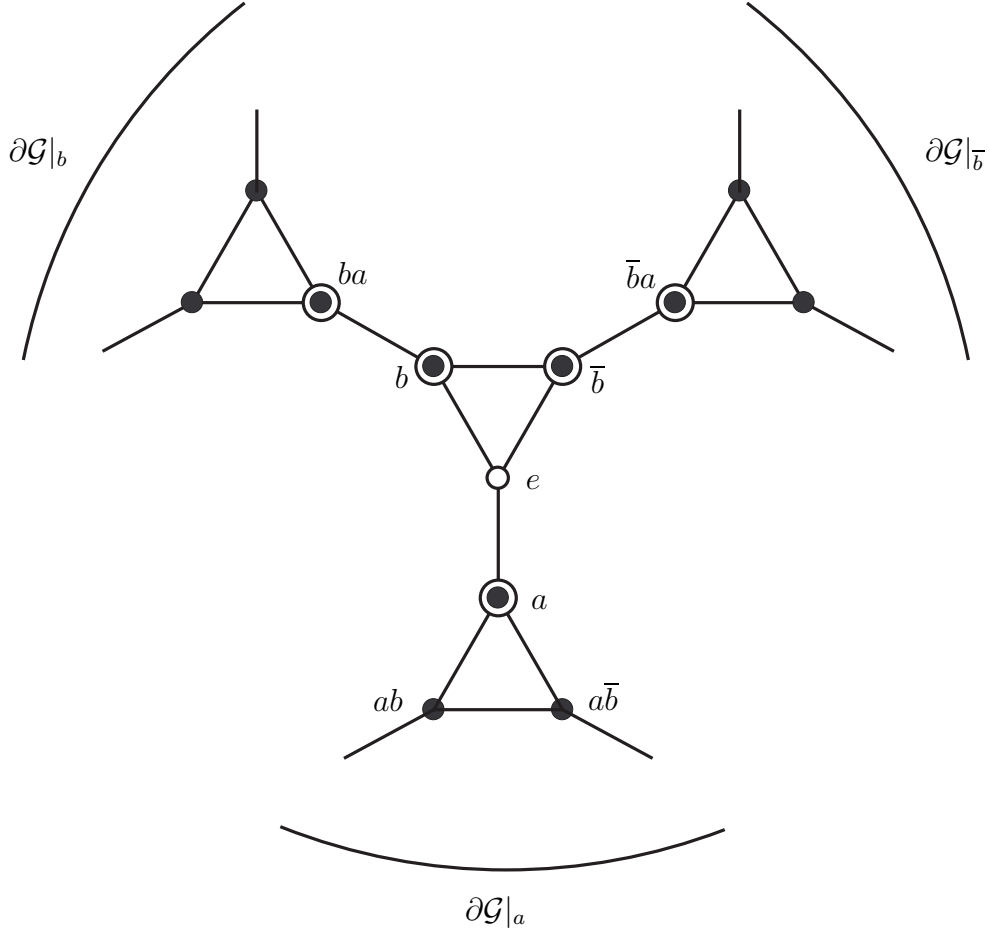}
\end{center}
\caption{The Cayley graph of $\ZZ_2*\ZZ_3$ and its boundary.}
\label{fig:graph}
\end{figure}

The set $\p\G$ of right infinite words in alphabet $\Ac$ subject to the same irreducibility rule as for $\G$ is endowed with the totally disconnected compact topology induced by the product topology on $\Ac^\infty$ (also see \remref{rem:tmc} and \secref{sec:H} below for a discussion of~$\p\G$ as a \emph{topological Markov chain} and as an \emph{ultrametric space}, respectively), and it is the \textsf{boundary} of a natural ``\textsf{word compactification}'' $\ol\G = \G \cup \p\G$ of~$\G$ (in our particular case this compactification coincides with the \emph{end} and the \emph{hyperbolic} compactifications of the group $\G$ as well). The left action of group $\G$ on itself extends to its continuous left action on the boundary $\p\G$. Akin to the definition of the group operation in $\G$, the latter action consists in concatenation followed, if necessary, by reduction.

By $\p\G|_g\subset\p\G$ we denote the \textsf{shadow} of a group element $g\neq e$, i.e., the \textsf{cylinder set} of infinite words that begin with $g$. The shadow
\begin{equation} \label{eq:aB}
\p\G|_a = (a\Bc)^\infty = \left\{ a\be_1a\be_2 a\ldots : \be_i\in\Bc \right\}
\end{equation}
is the space of infinite words ($\equiv$ \textsf{unilateral full shift space}) in alphabet $a\Bc = \left\{ab,a\ol b\right\}$,
or, in other words, is the boundary~$\p\Tc$ of the Cayley tree $\Tc = (a\Bc)^*$ of the free semigroup generated by $a\Bc$. Involution~$a$ establishes a bijection between the shadow~$\p\G|_a$ and its complement
\begin{equation} \label{eq:ap}
\p\G|'_a = \p\G \sm \p\G|_a
= \p\G|_b \sqcup \p\G|_{\ol b} = a\, \p\G|_a
= \left\{ \be_1a\be_2 a\ldots : \be_i\in\Bc \right\} \;.
\end{equation}

\begin{rem} \label{rem:gga}
When talking about shadows $\p\G|_g$, we will often assume (without any loss of generality, because $\p\G|_g=\p\G|_{ga}$ whenever $g\not\in\{e,a\}$) that $g$ belongs to the set
$$
\G|^a = \{g\in\G: \text{the last letter of}\; g\; \text{is}\; a \} \;.
$$
\end{rem}

\subsection{Modular group} \label{sec:modu}

We denote by $\Iso_0(\Hb^2)\cong\PSL(2,\RR)$ the group of orientation preserving isometries of the \emph{hyperbolic plane} $\Hb^2$ acting by fractional linear transformations in the \emph{upper complex half-plane model}:
\begin{equation} \label{eq:fl}
\begin{pmatrix}a & b\\ c & d \end{pmatrix}: z\mapsto \frac{az+b}{cz+d} \;.
\end{equation}
It is known since Dedekind and Klein that assignment of matrices
\begin{equation} \label{eq:s}
S= \begin{pmatrix} 0 & \!\!-1 \\ 1 & \!\!\phantom{-}0 \end{pmatrix}: z \mapsto -\frac{1}{z} \;, \qquad
T = \begin{pmatrix} 0 & \!\!-1 \\ 1 & \!\!\phantom{-}1 \end{pmatrix} : z\mapsto \frac{-1}{z+1}
\end{equation}
to the respective generators $a$ and $b$ of the groups $\ZZ_2$ and $\ZZ_3$ \eqref{eq:z23} establishes an isomorphism between the \textsf{modular group} $\PSL(2,\ZZ)\subset\PSL(2,\RR)$ and the free product $\G=\ZZ_2*\ZZ_3$.

The action of $\Iso_0(\Hb^2)$ on the hyperbolic plane extends to its continuous action on the boundary $\p\Hb^2$ of the standard \emph{visual compactification} $\ol{\Hb^2}$. In the upper half-plane model~$\p\Hb^2$ is identified with the extended real line $\ol\RR=\RR\cup\{\infty\}$, and the boundary action is given by the same formula \eqref{eq:fl} with $z\in\ol\RR$.

\subsection{Mediant tree} \label{sec:med}

Let
$$
R = ST = \begin{pmatrix} 1 & 1 \\0 & 1 \end{pmatrix}: z \mapsto z+1 \;, \qquad
L = ST^2 = \begin{pmatrix} 1 & 0 \\ 1 & 1 \end{pmatrix}: z\mapsto \frac{z}{z+1}
$$
be the matrices associated with the products $ab$ and $a\ol b$. The corresponding boundary maps establish orientation preserving homeomorphisms between the extended positive ray $\Ib=[0,\infty]=\ol\RR_+$ and its images $\Ib_L=L(\Ib) = [0,1]$ and $\Ib_R=R(\Ib) = [1,\infty]$; moreover, $\Ib_L \cup \Ib_R = \Ib$, and $\Ib_L \cap \Ib_R = \{1\}$. Note that $\Ib_L$ is on the \emph{left}, and $\Ib_R$ is on the \emph{right} with respect to the conventional direction of the real line ``from left to right'', which is the reason for the $(L,R)$ notation apparently first introduced by \textsc{Raney} \cite{Raney73}.

Further iteration of maps $L$ and $R$ gives rise to a binary tree of intervals $\Ib_w=w(\Ib)$ parameterized by finite words $w\in\{L,R\}^*$ and ordered by inclusion; for example, the interval
\begin{equation} \label{eq:LLR}
\Ib_{LLR}=\left[\tfrac13,\tfrac12\right] \subset \Ib_{LL}=\left[0,\tfrac12\right] \subset \Ib_L=[0,1] \subset \Ib
\end{equation}
corresponds to the word $LLR$, see \figref{fig:tree}. This tree is isomorphic to the Cayley tree of the free semigroup generated by $L$ and $R$, i.e., to the binary rooted tree $\Tc= (a\Bc)^*$ introduced in \secref{sec:cayley}. It can also be visualized by fixing the point $1\in\Ib$ (we remind that this is the intersection of the images $\Ib_L$ and $\Ib_R$), and representing each interval $\Ib_w$ by the point $w(1) = w(\Ib_L) \cap w(\Ib_R) = \Ib_{wL} \cap \Ib_{wR}$.
Actually, $w(1)$ is the \textsf{mediant} of the interval~$\Ib_w$ in the sense that it is the result of ``addition''\,\footnotemark
$$
w(1) = \frac{p_1}{q_1} \oplus \frac{p_2}{q_2} = \frac{p_1+p_2}{q_1+q_2} \;,
$$
of the irreducible fractions representing the endpoints of $\Ib_w$; for example, the mediant of the interval $\Ib_{LLR}=\left[\tfrac13,\tfrac12\right]$ \eqref{eq:LLR} is precisely
$$
\frac13 \oplus \frac12 = \frac{1+1}{2+3} = \frac25 = LLR(1) \;.
$$
We refer to this realization of the tree $\Tc$ as the \textsf{mediant tree}.

\footnotetext{\;It is commonly referred to as the ``Farey addition'', but is also known as the ``freshman'', or ``English major'' addition, according to \textsc{Kirillov} \cite[p. 99]{Kirillov13}. Its appearance in \textsc{Devaney} \cite{Devaney20} is accompanied with the following warning: \emph{Students are absolutely forbidden to use Farey addition in the real world. Only folks with a Ph.D.\ in mathematics are allowed to add fractions this way. Things become so much easier once you get your Ph.D.\ in math!}}

\begin{rem} \label{rem:cont}
Usually the mediant tree is given the names of \emph{Farey} or \emph{Stern -- Brocot}, e.g., see \textsc{Graham~-- Knuth~-- Patashnik} \cite[Section 4.5]{Graham-Knuth-Patashnik94}, \textsc{Lagarias -- Tresser} \cite{Lagarias-Tresser95}, or \textsc{Hatcher}'s book \cite{Hatcher22} (the front cover of which is adorned with the ``Farey tesselation'' of the hyperbolic disk). Our choice of terminology is due to the fact that none of Farey (1816), Stern (1858) and Brocot (1861) ever mentions trees (introduced by Cayley in 1857), whereas the underlying notion of \emph{mediant subdivision} has a much longer history traced back to the ancient Greek and Hellenistic mathematics, and associated, \emph{inter alia}, with the names of Parmenides, Plato, and Pappus, see \textsc{Fowler} \cite[Section~9.1(b)]{Fowler99}.

The approximation technique based on iteration of the mediant subdivision (i.e., on descending the associated \emph{binary search tree}, in modern language) was used already in the 17th century by Schwenter, Pell, and Wallis for \emph{Reductions of Fractions, or Proportions to smaller Terms, as near as it may be to the just Value} (this is the title of Chapter~X of Wallis' \emph{Treatise of Algebra}), see \textsc{Fowler} \cite{Fowler91} and \textsc{Cretney} \cite[Section~2.3]{Cretney14}. However, this method (\emph{indirecte \& fort laborieuse}, according to Lagrange) was later replaced with a more efficient one based on continued fractions (the very notion of which was prompted by this development). The first step was made by Huygens in~1680, which was further elucidated by Euler in 1737 \cite[p.~145]{Cretney14}, and completed by Lagrange who essentially put it in the modern form in his \emph{Additions} to the French 1774 translation of Euler’s \emph{Elements of Algebra} (it is in the final Remarque 22 of these \emph{Additions} that the aformentioned quote appears). Beginning from the 19th century the roles switch: now it is the framework of continued fraction that is used for talking about the mediant.

For more historical details we refer to the notes of \textsc{Bruckheimer~-- Arcavi} \cite{Bruckheimer-Arcavi95}, \textsc{Lagarias -- Tresser} \cite[Section~6]{Lagarias-Tresser95},
and \textsc{Knuth} \cite[the answer to Exercise 40 of Section 4.5.3 on pp. 654-655]{Knuth98}, as well as to \textsc{Guthery}'s book \cite{Guthery11} entirely devoted to several centuries of this \emph{enduring motif of mathematics} (to quote \cite[p. 192]{Fowler91}) that has, indeed, been the source of an incessant stream of mathematical endeavours.
\end{rem}

\begin{figure}[h]
\begin{center}
\includegraphics[scale=1.]{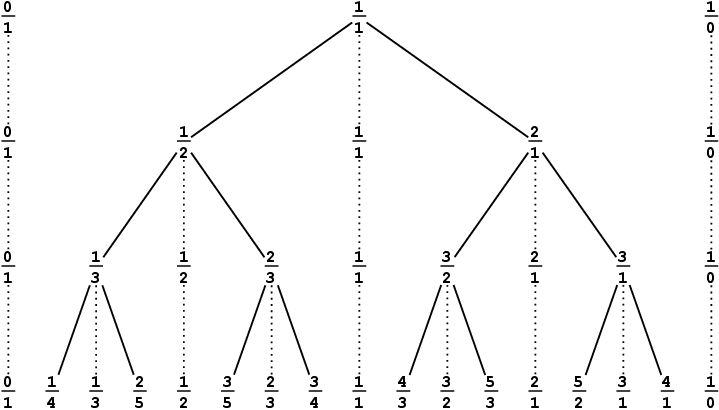}
\end{center}
\caption{The first three levels of the mediant tree.}
\label{fig:tree}
\end{figure}

\subsection{Boundary correspondence} \label{sec:bdry}

Since the mediant tree contains all positive rational numbers, passing to infinite words then produces a map
\begin{equation} \label{eq:LR}
\tau:\{L,R\}^\infty\to\Ib
\end{equation}
analogous to the ordinary binary expansion on the unit interval, namely, any irrational number has a unique presentation, whereas any positive rational number has two presentations: the ``left'' one ending with an infinite sequence of $R$'s and the ``right'' one ending with an infinite sequence of $L$'s (formally this statement goes back to \textsc{Hurwitz} \mbox{\cite[\S 5]{Hurwitz94}}, but in fact it is essentially equivalent to Euler's 1737 observation that any positive real number can be uniquely presented by a continued fraction).

Transformation $S$ \eqref{eq:s} is an orientation preserving homeomorphism between the negative and the positive rays of the real line, and therefore map \eqref{eq:LR} extends to an equivariant \textsf{boundary correspondence} (``encoding'')
\begin{equation} \label{eq:tau}
\tau:\p\G \to \p \Hb^2\cong \ol\RR \;,
\end{equation}
which is surjective, 1-to-1 on $\RR\sm\QQ$, and 2-to-1 on $\QQ\cup\{\infty\}$ (by a slight abuse of notation we use the same symbol for the original map \eqref{eq:LR} and its extension \eqref{eq:tau}).

Returning to hyperbolic plane $\Hb^2$ in the upper complex half-plane model, any finite word $w\in\{L,R\}^*$ sends the ``vertical'' geodesic joining the boundary points $0$ and~$\infty$ (i.e., the endpoints of $\Ib$) to the geodesic (Euclidean semi-circle) joining the endpoints of the interval $\Ib_w$, so that the ``right'' hyperbolic half-space that spans $\Ib$ is sent to the hyperbolic half-space that spans the interval $\Ib_w$. Elementary compactness considerations imply that if $\Ib_w\subset [0,1]=\Ib_L$, i.e., if $w$ begins with $L$, then the Euclidean length of~$\Ib_w$ uniformly goes to 0 as the word length $|w|$ of $w$ goes to infinity, whence the Euclidean distance between $w(i)$ (where $i$ is the imaginary unit which lies on the geodesic joining $0$ and~$\infty$) and $w(1)\in \Ib_w$ also uniformly goes to 0 as $|w|\to\infty$. Thus, $\tau(\om) = \lim_n [\om]_n (i)$ in the Euclidean topology for any infinite word $\om\in\{L,R\}^\infty$ beginning with $L$, where $[\om]_n\in \{L,R\}^n$ denotes the $n$-\textsf{truncation} of $\om$, i.e., the word that consists of the first $n$ letters of $\om$. Since the action of $\G$ on $\ol{\Hb^2}$ is continuous, one can then pass to the general situation and to claim that $\tau(\ga) = \lim_n [\ga]_n z$ in the topology of the compactification $\ol{\Hb^2}$ for any $\ga\in\p\G$ and any~$z\in\Hb^2$.

The above description of the boundary map $\tau$ \eqref{eq:tau} can also be presented in a more graphical form by using either of two prominent tessellations of the hyperbolic plane: the Dedekind -- Klein (or modular) one by translates of the standard fundamental domain of the group $\PSL(2,\ZZ)$, and the Smith -- Hurwitz (or Farey) one by ideal triangles, see \textsc{Lagarias -- Tresser} \cite[Section~5]{Lagarias-Tresser95}.

\subsection{Continued fractions} \label{sec:cf}

Another interpretation of the boundary correspondence described in \secref{sec:bdry} can be given in terms of \emph{continued fractions}, e.g., see \textsc{Borwein~-- van der Poorten -- Shallit -- Zudilin} \cite[Section 4.3]{Borwein-vanderPoorten-Shallit-Zudilin14}. Indeed, for any integer $n$
$$
R^n = \begin{pmatrix} 1 & n \\ 0 & 1 \end{pmatrix} = C_n J \;, \qquad
L^n = \begin{pmatrix} 1 & 0 \\ n & 1 \end{pmatrix} = J C_n \;,
$$
where
\begin{equation} \label{eq:cn}
C_n = \begin{pmatrix} n & 1 \\ 1 & 0 \end{pmatrix}: z\mapsto \frac{nz+1}z = n + \frac1z
\end{equation}
and
$$
J = C_0 = \begin{pmatrix} 0 & 1 \\ 1 & 0 \end{pmatrix}: z\mapsto \frac1z
$$
(note that $C_n$ and $J$ belong to $\GL(2,\ZZ)$, but not to $\SL(2,\ZZ)$, as their determinants are equal to $-1$), so that $\{C_n\}_{n=1}^\infty$ freely generate a free subsemigroup of $\GL(2,\ZZ)$.

Thus, for any $n,m\ge 0$
$$
R^n L^m = C_n J^2 C_m = C_n C_m \;,
$$
and for any not eventually constant infinite word $R^{n_1}L^{m_1}R^{n_2}L^{m_2}\dots$ its image under the map \eqref{eq:LR} is precisely the value of the infinite continued fraction
$$
n_1 + \dfrac1{m_1+\dfrac1{n_2+\dfrac1{m_2+\dfrac1\ddots}}}
= [n_1;m_1,n_2,m_2,\dots] \;,
$$
determined by the infinite product $C_{n_1}C_{m_1}C_{n_2}C_{m_2}\cdots$. In this presentation an infinite word in the alphabet $\{L,R\}$ is considered as an intermittent sequence of ``syllables'' consisting of repetitions of the same letter. The convergents of the associated infinite continued fractions then correspond to the iterations of the Huygens -- Euler -- Lagrange approximation mentioned in \remref{rem:cont} above.

\subsection{Minkowski and Denjoy measures} \label{sec:MD}

Given a parameter $0<\al<1$, we denote by~$\ka^\al$ the probability measure on the boundary component $\p\G|_a\cong (a\Bc)^\infty\subset\p\G$,
see~\eqref{eq:aB}, obtained from the Bernoulli measure on~$\Bc^\infty$ with the base distribution $(\al,1-\al)$. In probabilistic terms $\ka^\al$ is the distribution of the random infinite words $a\be_1a\be_2 a\cdots$, where~$\be_i$ are independent identically distributed and take values $b$ and $\ol b$ with the respective probabilities $\al$ and $1-\al$. The image $a\ka^\al$ of measure $\ka^\al$ under the involution~$a$ is concentrated on the complement $\p\G|'_a$ \eqref{eq:ap}, and it is the distribution of random infinite words $\be_1a\be_2 a\cdots$.

Given another parameter $0<p<1$, we denote by
\begin{equation} \label{eq:kaalp}
\ka^{\al,p} = p\ka^\al + (1-p) a\ka^\al
\end{equation}
the convex combination of the probability measures $\ka^\al$ and~$a\ka^\al$ with the respective weights $p$ and $1-p$. More explicitly, the measures $\ka=\ka^{\al,p}$ can be described by their values
\begin{equation} \label{eq:kag}
\ka_g = \ka\left(\p\G|_g\right) =
\begin{cases}
p \al_{\be_1} \dots \al_{\be_n} \;,& g=a\be_1a\be_2a\dots\be_na \;, n\ge 0 \;, \\
(1-p) \al_{\be_1} \dots \al_{\be_n} \;,& g=\be_1a\be_2a\dots\be_na \;, n\ge 1 \;,
\end{cases}
\end{equation}
on shadow (cylinder) sets $\p\G|_g$, where $\be_i\in\Bc=\left\{b,\ol b\right\}$,
$$
\al_b=\al \;, \qquad  \al_{\ol b}=(1-\al) \;,
$$
and we assume that $g\in\G|^a$, see \remref{rem:gga}.

Actually, below it will be more convenient to use new parameters
\begin{equation} \label{eq:pi}
\begin{cases}
\begin{aligned}
&\pi_a = \displaystyle\frac{p}{1-p} = \frac{\ka_a}{1-\ka_a} \;, \\[1ex]
&\pi_{\be a} = \al_\be = \displaystyle\frac{\ka_{\be a}}{1-\ka_a} \;,\quad\be\in\Bc \;,
\end{aligned}
\end{cases}
\end{equation}
subject to condition $\pi_{ba} + \pi_{\ol b a}=1$, for describing the family $\left\{\ka^{\al,p}\right\}$. Conversely, one recovers $\al$ and $p$ as
\begin{equation} \label{eq:converse}
\begin{cases}
\begin{aligned}
&\al = \pi_{ba} \;, \\[1ex]
&p = \displaystyle\frac{\pi_a}{1+\pi_a} \;.
\end{aligned}
\end{cases}
\end{equation}
Expressions \eqref{eq:kag} for the measures of shadow sets then take the form
\begin{equation} \label{eq:kag1}
\ka_g = \frac{\pi_g}{1+\pi_a} \qquad \forall\, g\in \G|^a
\end{equation}
in terms of parameters \eqref{eq:pi}, where for $g\in\G|^a$ we multiplicatively put
\begin{equation} \label{eq:pmult}
\pi_g =
\begin{cases}
\pi_a \pi_{\be_1a} \dots \pi_{\be_na} \;,& g=a\be_1a\be_2a\dots\be_na \;, n\ge 0 \;, \\
\pi_{\be_1a} \dots \pi_{\be_na} \;,& g=\be_1a\be_2a\dots\be_na \;, n\ge 1 \;.
\end{cases}
\end{equation}

Measures $\ka=\ka^{\al,p}$ \eqref{eq:kaalp} are quasi-invariant with respect to the action of group $\G$ on the boundary $\p\G$, and by using formulas \eqref{eq:kag1}, \eqref{eq:pmult} one can easily write down the corresponding Radon -- Nikodym derivatives explicitly. For the generators of group $\G$
$$
\frac{da\ka}{d\ka}(\ga) =
\begin{cases}
\begin{aligned}
\displaystyle\frac1{\pi_a} &\;, \; \ga\in \p\G|_a \;, \\[1ex]
\pi_a &\;, \; \ga\in \p\G|'_a \;,
\end{aligned}
\end{cases}
\quad\text{and}\qquad
\frac{db\ka}{d\ka}(\ga) =
\begin{cases}
\begin{aligned}
\displaystyle\frac{\pi_{\ol b a}}{\pi_a} &\;, \; \ga\in \p\G|_a \;, \\[1ex]
\displaystyle\frac{\pi_a}{\pi_{ba}} &\;, \; \ga\in \p\G|_b \;, \\[1ex]
\displaystyle\frac{\pi_{ba}}{\pi_{\ol b a}} &\;, \; \ga\in \p\G|_{\ol b} \;.
\end{aligned}
\end{cases}
$$

More generally, given $g_1,g_2\in\G$ and $\ga\in\p\G$, let $g_1 \bw g_2$ denote the \textsf{$a$-confluent} of the geodesic rays $[g_1,\ga)$ and $[g_2,\ga)$, i.e., the endpoint of their first common $a$-labelled edge that is the ``closest'' to $\ga$ , so that both $g_1^{-1} \left(g_1 \bw g_2\right)$ and $g_2^{-1} \left(g_1 \bw g_2\right)$ belong to $\G|^a$, see \figref{fig:conf}. Then
\begin{equation} \label{eq:dgka}
\frac{dg\ka}{d\ka}(\ga) = \frac{\pi\left[g,e \bw g \right]}{\textstyle\pi\left[e,e \bw g \right]} \;,
\end{equation}
where we denote by $\pi[g_1,g_2]=\pi_{g_1^{-1}g_2}$ the product of weights from collection~\eqref{eq:pi} along the geodesic $[g_1,g_2]$ whenever $g_1^{-1}g_2\in\G|^a$, i.e., whenever the last segment of the geodesic $[g_1,g_2]$ is labelled with $a$. In other words, \eqref{eq:dgka} is the natural regularization of the formal ratio
\begin{equation} \label{eq:formal}
\mathquote{\;\frac{\pi[g,\ga)}{\pi[e,\ga)}\;}
\end{equation}
of the infinite products of weights \eqref{eq:pi} along the geodesic rays $[g,\ga)$ and $[e,\ga)$.

\begin{figure}[h]
\begin{center}
\psfrag{a}[cl][cl]{$g_2$}
\psfrag{d}[cl][cl]{$g_1$}
\psfrag{f}[cr][cr]{$g_1 \bw g_2$}
\psfrag{g}[cr][cr]{$\ga$}
\includegraphics[scale=.6]{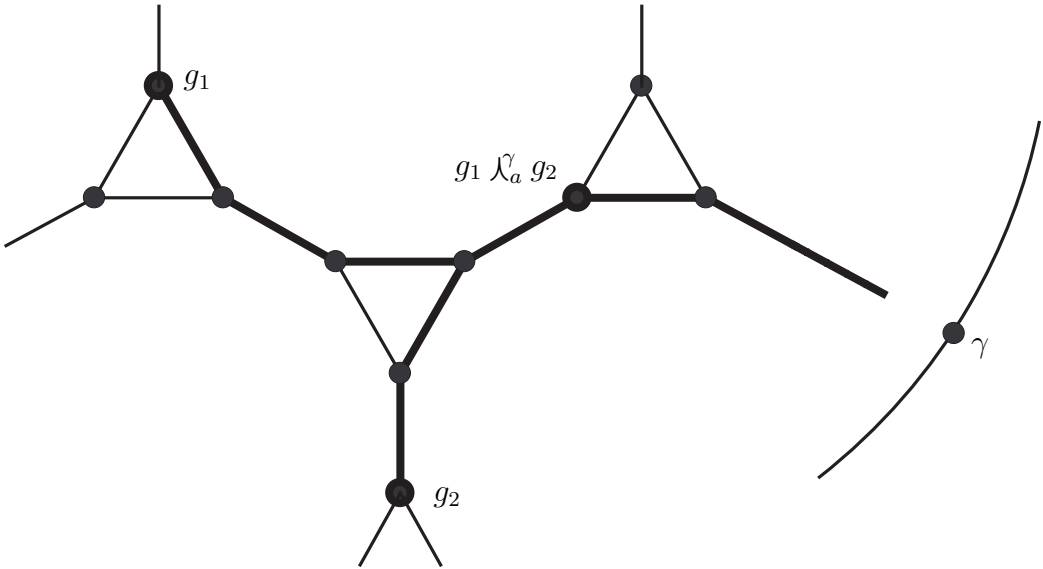}
\end{center}
\caption{The $a$-confluent of two geodesic rays.}
\label{fig:conf}
\end{figure}

For a fixed $0<\al<1$ we denote by $\kab^\al$ the common quasi-invariant class of the measures $\left\{\ka^{\al,p}: 0<p<1\right\}$ on $\p\G$, and call it the \textsf{Denjoy measure class} with parameter~$\al$. For $\al=\frac12$ we refer to it as the \textsf{Minkowski measure class}.

\begin{rem} \label{rem:tmc}
The boundary $\p\G$ considered as a set of infinite irreducible words in alphabet $\Ac=\left\{a,b,\ol b\right\}$ is the \emph{topological Markov chain} on the state space $\Ac$ with the admissible transitions $a \rightsquigarrow b,\ol b$ and $b,\ol b\rightsquigarrow a$. For any \textsf{base} probability measure $\si$ on $\Ac$, the associated \textsf{multiplicative Markov measure} $m_\si$ on $\p\G$ is the Markov measure with the initial distribution $\si$ and the transition probabilities which are the normalized restrictions of the base measure to the corresponding sets of states (the transitions to which from a given departure state are admissible), see \textsc{Mairesse} \cite{Mairesse05}, \textsc{Kaimanovich} \cite{Kaimanovich25}, and the references therein for further details. One can immediately see that $m_\si=\ka^{\al,p}$ with the parameters
$$
\al=\frac{\si(b)}{\si(b)+\si\left(\ol b\right)} \;, \qquad p=\si(a)=1-\si(b)-\si\left(\ol b\right) \;,
$$
and, conversely, any measure $\ka^{\al,p}$ is multiplicative Markov with the base distribution
$$
\si(a) = p \;, \qquad \si(b)= (1-p)\al \;, \qquad \si\left(\ol b\right)=(1-p)(1-\al) \;.
$$
\end{rem}

\begin{rem}
The image $\tau\left(\ka^{\frac12}\right)$ of measure $\ka^{\frac12}$ under boundary map $\tau$ \eqref{eq:tau} was first considered by \textsc{Minkowski} \cite{Minkowski05}. More precisely, his \textsf{\mbox{Fragefunktion}}~\textbf{?} is the cumulative distribution function of the normalized restriction of the image measure $\tau\left(\ka^{\frac12}\right)$ to the unit interval. In terms of continued fractions (see \secref{sec:cf}) the aforementioned restriction is the distribution of the random infinite fraction $[0;m_1,m_2,\dots]$ with independent $\Geom\left(\frac12\right)$-distributed denominators $m_i$, cf.\ \remref{rem:cf} below.

For an arbitrary value $0<\al<1$ the measures $\tau(\ka^\al)$ were introduced by \textsc{Denjoy} \cite{Denjoy34, Denjoy38} (again in terms of their distribution functions) who proved their singularity to the Lebesgue measure (at that time examples of full support singular measures were still a novelty).\footnotemark\ The Minkowski function and the associated measure have gained enormous popularity in both mathematical and paramathematical literature (e.g., see \textsc{Salem} \cite{Salem43}, \textsc{de Rham} \cite{deRham57}, \textsc{Kinney} \cite{Kinney60} \textsc{Conway} \cite[Chapter 8]{Conway01}, \textsc{Mandelbrot} \cite{Mandelbrot04}, \textsc{Kesseb\"{o}hmer -- Stratmann} \cite{Kessebohmer-Stratmann08}, or \textsc{Kirillov} \cite{Kirillov13}, and the numerous references therein), and the graph of the function\ \textbf{?}\ was nicknamed \emph{slippery devil's staircase} by \textsc{Gutzwiller -- Mandelbrot} \cite{Gutzwiller-Mandelbrot88}. Although the work of Denjoy is profusely quoted in the literature on the Minkowski function, the general Denjoy measures have remained much less known than the Minkowski one, to the extent that in 1995 they were reinvented by \textsc{Tichy~-- Uitz}~\cite{Tichy-Uitz95}.
\end{rem}

\footnotetext{\;Minkowski introduced his function for entirely number theoretical purposes, and its singularity was first addressed by Denjoy who learned about Minkowski's function in 1931 from Hadamard, \emph{dont l'érudition si vaste et toujours si opportunément présente m'a suggéré ce remarquable sujet d'étude}, to quote the acknowledgement from \cite{Denjoy38}. }

\begin{rem} \label{rem:rot}
It is easy to see that the measure $\ka^{\frac12}$, its translate $a\ka^{\frac12}$, and their convex combinations $\ka^{\frac12,p}$ are the only probability measures on the boundary $\p\G$ invariant with respect to the group $\Rc$ of ``rotations'' of the unlabelled Cayley graph (i.e., with respect to all automorphisms of the unlabelled graph that fix the identity vertex).
\end{rem}

\subsection{Radon -- Nikodym problem} \label{sec:RN}

The \textsf{Radon -- Nikodym problem} for a real valued multiplicative cocycle of a group action consists in realizing this cocycle as the Radon~-- Nikodym cocycle of an appropriate quasi-invariant measure on the action space, see \textsc{Schmidt} \cite[\S 10]{Schmidt77}, \textsc{Renault} \cite{Renault05}, or \textsc{Kaimanovich} \cite{Kaimanovich25}. This problem plays a central role in the theory of Gibbs measures\,---\,which are essentially defined by specifying their Radon -- Nikodym derivatives.

For any prescribed set
\begin{equation} \label{eq:pi0}
\pi=\left(\pi_a,\pi_{ba},\pi_{\ol b a}\right)
\end{equation}
of positive values of weights \eqref{eq:pi} the right-hand side of formula \eqref{eq:dgka} defines a real-valued multiplicative \textsf{product cocycle}
\begin{equation} \label{eq:pc}
\xi (g,\ga) = \xi_\pi (g,\ga)
\end{equation}
of the boundary action of group $\G$. As we have already pointed out, geometrically $\xi$ is the natural regularization of ratios \eqref{eq:formal} of the infinite products of weights \eqref{eq:pi0} along the geodesic rays in the Cayley graph of group $\G$ issued from different initial points towards the same boundary point.

\begin{prp} \label{prp:RN}
The Radon -- Nikodym problem for cocycle $\xi_\pi$ \eqref{eq:pc} of the boundary action of group $\G$ determined by a collection of positive weights \eqref{eq:pi0} is solvable if and only if $\pi_{ba}+\pi_{\ol b a}=1$, and under this condition the measure $\ka^{\al,p}$ with the parameters determined by formula \eqref{eq:converse} is its unique normalized solution.
\end{prp}

\begin{proof}
Indeed, let $\nu$ be a quasi-invariant probability measure on the boundary $\p\G$ with the Radon -- Nikodym cocycle $\xi_\pi$. Then by formula \eqref{eq:dgka}
$$
\frac{d g^{-1}\nu}{d\nu}(\ga) = \pi_g \qquad \forall\, g\in\G|^a, \; \ga\in \p\G|'_a \;,
$$
whence
\begin{equation} \label{eq:nugg}
\nu (\p\G|_g ) = \nu (g\,\p\G|'_a ) = g^{-1} \nu (\p\G|'_a) = \pi_g \, \nu (g\,\p\G|'_a ) \;,
\end{equation}
so that the Kolmogorov consistency condition on the measures of cylinder sets \eqref{eq:nugg} implies that $\pi_{ba}+\pi_{\ol b a}=1$. Then comparing \eqref{eq:nugg} with \eqref{eq:kag1} yields the claim.
\end{proof}

\subsection{Hausdorff measure} \label{sec:H}

We remind that the $d$-dimensional (outer) \textsf{Hausdorff measure} on a metric space $X$ is defined by putting
$$
\Hc^d(A) = \lim_{\ep\to 0} \inf_\Uc \sum_n (\diam U_n)^d
$$
for any subset $A\subset X$, where the infimum is taken over all countable covers $\Uc=(U_n)$ of~$A$ by sets with $\diam U_n\le\ep$. Its restriction to the Borel $\si$-algebra is a \emph{bona fide} measure (possibly zero or infinite, though). The \textsf{Hausdorff dimension} of $X$ is
$$
\dim_\text{H} X = \inf\{d>0: \Hc^d(X)=0\} = \sup\{d>0: \Hc^d(X)=\infty\} \;.
$$
If $X$ is an \emph{ultrametric} space, and $m$ is a \textsf{well-scaling} Borel measure on $X$ in the sense that $m(B) = (\diam B)^d$ for any ball $B\subset X$, then one can easily see that $\dim_{\text H} X=d$ and $\Hc^d=m$.

The boundary $\p\G$ is endowed with the natural ultrametric
\begin{equation} \label{eq:rho}
\rho(\ga_1,\ga_2) = e^{-(\ga_1|\ga_2)} \;, \qquad\ga_1,\ga_2\in\p\G \;,
\end{equation}
where the \textsf{Gromov product} $(\ga_1|\ga_2)$ is the length of the common part (\textsf{confluence}) of the geodesic rays $[e,\ga_1)$ and $[e,\ga_2)$. The balls of metric $\rho$ are precisely the shadows $\p\G|_g$ with $g\in\G|^a$, and $\diam \p\G|_g = e^{-|g|}$ for any $g\in\G|^a$.

By comparing the latter formula with the descriptions of measure $\ka^\al$ and of its translate~$a\ka^\al$ from \secref{sec:MD}, one can conclude, in view of the above observation about well-scaling measures on ultrametric spaces, that the Hausdorff dimension of both the subset $\p\G|_a\subset\p\G$ and its complement $\p\G|_a'=a\p\G|_a$ is $d=\log 2/2$, and that the associated Hasudorff measures are precisely $\frac1{\sqrt2}\ka^{\frac12}$ and $a\ka^{\frac12}$, respectively. Therefore,
$$
\tfrac{\sqrt2}2 \ka^{\frac12} + a\ka^{\frac12} = \left( \tfrac{\sqrt2}2 + 1 \right)\ka^{{\sst{\frac12}},p} \;,
$$
where $p=\frac1{1+\sqrt2}$, is the Hausdorff measure on the whole boundary $\p\G$. The parameters~\eqref{eq:pi} describing the above measure $\ka^{{\sst{\frac12}},p}$ are $\pi_a = \frac{\sqrt2}2$ and $\pi_{ba} = \pi_{\ol b a} = \frac12$, so that by formula \eqref{eq:pmult} $\pi_g = \left( \frac{\sqrt2}2 \right)^{|g|}$ for all $g\in\G|^a$, and the logarithm of the corresponding cocycle $\xi=\xi_\pi$ \eqref{eq:pc} of the boundary action is proportional to the additive \emph{Busemann cocycle} of the boundary action with respect to the word metric, the proportionality coefficient being the \emph{exponential growth rate} of $\G$ (cf.\ \textsc{Kaimanovich} \cite[Sections 1.4, 3.5]{Kaimanovich90} and \textsc{Ledrappier} \cite[Section 1.a]{Ledrappier01}).

\section{Harmonic measure} \label{sec:2}

\subsection{Random walks on the modular group} \label{sec:rwmod}

\textsf{Random walk} $(G,\mu)$ on a countable group $G$ determined by a probability measure (\textsf{step distribution}) $\mu$ is the Markov chain with the state space $G$ and \textsf{transitions} $\displaystyle g \mapstoto^{h\sim \mu} gh$ that consist in right multiplication by a $\mu$-distributed \textsf{increment}~$h$, so that the time $n$ position of a sample path issued at time~0 from the group identity is the product $g_n = h_1 h_2 \dots h_n$
of $n$ independent $\mu$-distributed increments~$h_i$, and the time $n$ distribution of the random walk is the $n$-fold convolution~$\mu^{*n}$ of the measure $\mu$.

Below we will be considering the random walks on the group $\G=\ZZ_2*\ZZ_3\cong\PSL(2,\ZZ)$ whose step distributions $\mu$ are subject to the following condition:

\vspace{-2mm}
\begin{equation} \label{eq:c}
\begin{minipage}{13cm}
\emph{The support $\supp\mu$ is contained in the set $\Sc = \bigl\{ a,b,\ol b,ba,\ol b a \bigr\}$, and the measure $\mu$ is \textsf{non-degenerate} in the sense that $\supp\mu$ generates $\G$ as a semigroup.}
\end{minipage}
\vspace{2mm}
\end{equation}
The non-degeneracy condition means that the random walk $(\G,\mu)$ visits any group element with a positive probability. It is easy to see that a measure $\mu$ with $\supp\mu\subset\Sc$ is non-degenerate if and only if $\supp\mu$ is not entirely contained in one of the three subsets $\{a\}, \bigl\{b,\ol b\bigr\}, \bigl\{ba, \ol b a\bigr\}$ of $\Sc$.

The elements of set $\Sc$ are marked $\cb$ in \figref{fig:graph}; together with group identity~$e$ they form a ``spiked triangle'' at the origin of the Cayley graph of $\G$. In principle, we could allow the measure $\mu$ to charge $e$ as well (cf.\ the proof of \tmeref{tme:ex2} below); however, it would not change the principal object of our study\,---\,the class of the arising harmonic measures on the boundary $\p\G$ (see \secref{sec:HarD}). Indeed, the harmonic measure of such a ``lazy'' step distribution is the same as for the step distribution obtained by removing the atom at $e$ with subsequent normalization.

Although the random walk $(\G,\mu)$ specified by condition \eqref{eq:c} is \emph{not} necessarily a nearest neighbour one (as $ba,\ol ba\in\Sc$ both have length 2), the definition of set $\Sc$ guarantees that any ``bridge crossing''\,---\,we remind that ``bridges'' are the edges of the Cayley graph of~$\G$ labelled with generator $a$ of $\ZZ_2$, see \figref{fig:graph}\,---\,made by the random walk ends on the opposite side right at the foot of the bridge, which plays a crucial role in our further considerations.

\subsection{Harmonic measures and Denjoy classes} \label{sec:HarD}

\begin{prp} \label{prp:st}
Under condition \eqref{eq:c}, sample paths of the random walk~$(\G,\mu)$ almost surely converge to the boundary~$\p\G$, and therefore the corresponding one-dimensional distributions (i.e., the \mbox{$n$-fold} convolutions $\mu^{*n}$) converge to the resulting \textsf{harmonic (hitting, limit)} probability measure $\nu=\nu(\mu)$ on~$\p\G$ in the weak* topology of the compactification $\ol\G = \G \cup \p\G$. The measure $\nu$ is the \emph{unique} $\mu$-\textsf{stationary} probability measure on~$\p\G$, i.e., the unique one that is preserved by convolution $\mu*\nu = \sum \mu(g) g\nu$ with measure~$\mu$.
\end{prp}

This is a particular case of a number of results on boundary convergence of random walks on groups due, in increasing generality, to \textsc{Cartwright -- Soardi} \cite{Cartwright-Soardi89}, \textsc{Woess} \cite[Theorem 3.3]{Woess89}, and \textsc{Kaimanovich} \cite[Theorem 7.4 and Theorem~7.6]{Kaimanovich00a}, the idea of which goes back to \textsc{Furstenberg} \cite[proof of Proposition 3.2]{Furstenberg71}, \cite[proof of Theorem 16.1]{Furstenberg73} (its broad scope was first pointed out by Margulis, see \textsc{Kaimanovich~-- Vershik} \cite[Section 6.8]{Kaimanovich-Vershik83}). Since the measure $\mu$ is finitely supported, boundary convergence alone also directly follows just from the fact that the random walk $(G,\mu)$ is transient (for instance, because\,---\,although this is really an overkill\,---\,the group generated by $\supp\mu$ is non-amenable), see \textsc{Picardello -- Woess} \cite[Section~5]{Picardello-Woess87} and \textsc{Woess} \cite[Lemma 6.1]{Woess89}.

Given a group element $g\in\G$, we denote by $\pi_g$ the corresponding \textsf{passage probability}, i.e., the probability that a sample path issued from group identity $e$ ever visits $g$, so that $\pi_e=1$, and $0<\pi_g<1$ for all other group elements (because of transience of the random walk). This notation has already been introduced in a different context in \secref{sec:MD} when talking about the Denjoy measures, and the following result shows that these two usages are consistent.

\begin{thm} \label{thm:ka}
Under condition \eqref{eq:c}, the harmonic measure $\nu$ of the random walk $(\G,\mu)$ belongs to one of the Denjoy measure classes $\kab^\al$. Moreover,
$$
\nu=\ka^{\al,p} \qquad \text{with} \quad \al=\pi_{ba} \quad\text{and}\quad p=\frac{\pi_a}{1+\pi_a} \;.
$$
\end{thm}

\begin{proof}
In order to identify the measure $\nu$, we have to find its values $\nu_g=\nu(\p\G|_g)$ on the shadow sets~$\p\G|_g$, and in view of \remref{rem:gga} we may assume that $g\in\G|^a$. Then by the aforementioned ``bridge crossing'' property, in order to end up in $\p\G|_g$, a sample path must first visit $g$, and then, starting from $g$, end up in~$\p\G|_g$, so that
\begin{equation} \label{eq:nug}
\nu_g = \pi_g \cdot g\nu(\p\G|_g) = \pi_g \cdot \nu(g^{-1}\p\G|_g) = \pi_g \cdot \nu (\p\G|'_a) = \pi_g (1-\nu_a) \;.
\end{equation}
In particular, $\nu_a = \pi_a (1-\nu_a)$ for $g=a$, whence $\nu_a = \pi_a/(1+\pi_a)$, so that formula \eqref{eq:nug} takes the form
\begin{equation} \label{eq:nug2}
\nu_g = \frac{\pi_g}{1+\pi_a} \qquad\forall\,g\in\G|^a \;.
\end{equation}
Note that applying \eqref{eq:nug2} to $g=\be a$ with $\be\in\Bc=\left\{b,\ol b\right\}$ then automatically yields
\begin{equation} \label{eq:11}
\pi_{ba} + \pi_{\ol b a} = (1+\pi_a)\left( \nu_{ba} + \nu_{\,\ol b a} \right)
= (1+\pi_a) (1-\nu_a) = 1
\end{equation}
(also see the proof of \thmref{thm:u} below).

In order to arrive at a point $g=a\be_1a\be_2a\dots\be_na$ (where $\be_i\in\Bc$), the random walk\,---\,again by the bridge crossing property\,---\,must consecutively visit all points $a, a\be_1a, a\be_1 a\be_2 a$, etc., so that the corresponding passage probability splits into the product
\begin{equation} \label{eq:piga}
\pi_g = \pi_a \pi_{\be_1 a}\cdots \pi_{\be_na} \;,
\end{equation}
and, in the same way, $\pi_g = \pi_{\be_1 a} \cdots \pi_{\be_n a}$ for $g=\be_1a\be_2 a\dots\be_n a$. Therefore, \eqref{eq:nug2} is precisely the description \eqref{eq:kag1} of the measure $\ka^{\al,p}$ with the claimed parameters $\al$ and $p$.
\end{proof}

\begin{cor}
The parameter $p$ belongs to the open interval $\left(0,\frac12\right)$.
\end{cor}

\subsection{Equations for the passage probabilities} \label{sec:pass}

Now we have to investigate the dependence of the parameters $\al,p$ describing the harmonic measure $\nu=\ka^{\al,p}$ in \thmref{thm:ka} (or, equivalently, of the passage probabilities $\pi_a,\pi_{ba}$) on step distribution $\mu$.

We begin with the following classical argument. In order to arrive at a point $g\neq e$ starting from group identity $e$, random walk~$(\G,\mu)$ has to make the first step to a $\mu$-distributed point $h$, after which it has to travel from $h$ to $g$, which is the same as travelling from $e$ to $h^{-1} g$. Thus, the corresponding passage probability $\pi_g$ must satisfy the \textsf{$\mu$-stationarity} condition
$$
\pi_g = \sum_h \mu(h) \pi_{h^{-1} g} \qquad\forall\,g\in\G\sm\{e\}\;.
$$

Due to condition \eqref{eq:c}, for our purposes it is sufficient to consider these equations just for three group elements $g=a,ba,\ol b a$ with respect to three real variables
\begin{equation} \label{eq:xyz}
x = \pi_a, \; y = \pi_{ba}, \; \ol y = \pi_{\ol b a} \;,
\end{equation}
for the sake of symmetry forgetting, for the time being, about already established relation~\eqref{eq:11}.

Taking into account that $\pi_{aba}=\pi_a\pi_{ba}$ and $\pi_{a{\ol b}a}=\pi_a\pi_{{\ol b}a}$ by multiplicativity property~\eqref{eq:piga}, and that $\pi_e=1$, we obtain then the \textsf{master system of equations}
\begin{equation} \label{eq:sys}
\begin{cases}
x = \af + \bf\, \ol y + \ol\bf\, y + \bf'\, x\ol y + \ol\bf'\, xy \\
y = \af\, xy + \bf\, x + \ol\bf\, \ol y + \bf' + \ol\bf'\, x \ol y \\
\ol y = \af\, x\ol y + \bf\, y + \ol\bf\, x  + \bf'\, x y + \ol\bf' \;,
\end{cases}
\end{equation}
where we use the shorthand notation
\begin{equation} \label{eq:w}
\mu(a) = \af, \; \mu(b) = \bf, \; \mu\left( \ol b \right) = \ol\bf, \; \mu(ba) = \bf', \;
\mu\left( \ol b a \right) = \ol\bf'
\end{equation}
for the weights of measure $\mu$.

\begin{rem} \label{rem:swap}
The whole system \eqref{eq:sys} is symmetric with respect to \mbox{``involution''~$\cdot \leftrightarrow \ol\cdot$}, which corresponds to the fact that swapping $b$ and $\ol b$ results in an automorphism of group~$\G$.
\end{rem}

\begin{thm} \label{thm:u}
Under condition \eqref{eq:c}, passage probabilities \eqref{eq:xyz} are the unique solution of master system \eqref{eq:sys} in the open unit cube
\begin{equation} \label{eq:pm}
0<x,y,\ol y<1 \;.
\end{equation}
This solution admits the following explicit description:
\begin{enumerate}
\item[$\bullet$]
$y=\pi_{ba}$ is the unique solution of equation
\begin{equation} \label{eq:y}
\begin{aligned}
&\left[ y - \ol\bf(1-y) - \bf' \right] \cdot \left[ \af (1-y) + \bf'\, y + \ol\bf \right] \\
&= \left[ \af\,y + \ol\bf' (1- y) + \bf \right] \cdot \left[ (1-y) - \bf\,y - \ol\bf' \right]
\end{aligned}
\end{equation}
in the open unit interval $(0,1)$;
\item[$\bullet$]
$\ol y=\pi_{\ol b a}=1-y$ for $y=\pi_{ba}$;
\item[$\bullet$]
$x=\pi_a$ is expressed in terms of $y=\pi_{ba}$ and $\ol y=\pi_{\ol b a}=1-y$ as
\begin{equation} \label{eq:x}
x = \frac{1 - \bf\,y - \ol\bf\,\ol y - \bf' - \ol\bf'}{1 -\bf'\,\ol y -\ol\bf'\,y} \;.
\end{equation}
\end{enumerate}
\end{thm}

\begin{proof}
(I) \emph{Uniqueness.} Any triple $\left(x,y,\ol y\right)$ gives rise to the associated product cocycle $\xi$ \eqref{eq:pc}, and system \eqref{eq:sys} rewrittent as
\begin{equation} \label{eq:rer}
\begin{cases}
\begin{aligned}
\displaystyle &1 = \af\, \frac1x + \bf\, \frac{\ol y}x + \ol\bf\, \frac{y}x + \bf'\, \ol y + \ol\bf'\, y \\[1ex]
\displaystyle &1 = \af\, x + \bf\, \frac{x}y + \ol\bf\, \frac{\ol y}y + \bf'\,\frac1y + \ol\bf'\, \frac{x \ol y}y \\[1ex]
\displaystyle &1 = \af\, x + \bf\, \frac{y}{\ol y} + \ol\bf\, \frac{x}{\ol y}  + \bf'\, \frac{x y}{\ol y} + \ol\bf'\, \frac1{\ol y} \;,
\end{aligned}
\end{cases}
\end{equation}
is \textsf{$\mu$-harmonicity} condition
$$
1 = \sum_h \mu(h) \xi \left( h^{-1},\ga \right) \qquad \forall\,\ga\in\p\G
$$
on cocycle $\xi$. Indeed, the above condition for $\ga$ from each of three subsets $\p\G|_a,\p\G|_{ba}$, and $\p\G|_{\ol b a}$ of $\p\G$ yields precisely one of the three equations of system~\eqref{eq:rer}.

If cocycle $\xi$ is realized as the Radon -- Nikodym cocycle of a quasi-invariant probability measure on $\p\G$, then $\mu$-harmonicity of the cocycle is equivalent to $\mu$-stationarity of the measure. We know from \prpref{prp:RN} that the Radon -- Nikodym problem for cocycle~$\xi$ is solvable if and only if $y+\ol y=1$. Therefore, the uniqueness part of \prpref{prp:st} implies that there is a unique solution of \eqref{eq:sys} with $y+\ol y=1$, and it remains to show that there are no solutions with $y+\ol y\neq 1$.

Let $t=y+\ol y$. The sum of the equations from system \eqref{eq:sys} is then
$$
x + t = \af + \bf' + \ol\bf' + \left( \bf + \ol\bf \right) x + \left( \bf + \ol\bf + \left( \af + \bf' + \ol\bf \right)x \right) t \;,
$$
or, since the sum of weights \eqref{eq:w} of measure $\mu$ is 1,
$$
x+t = 1 - B + B x + \left( B + \left( 1 -B \right)x \right) t \;,
$$
where $B=\bf+\ol\bf$, and, finally, $(x-1)(t-1)(B -1)=0$, whence $t=1$, because $B<1$ by condition \eqref{eq:c}, and we are only interested in the solutions with $x<1$.

\smallskip

(II) \emph{Explicit solution.} As we have seen in the first part of the proof, on the open unit cube system \eqref{eq:sys} is equivalent to system
\begin{equation} \label{eq:sys1}
\begin{cases}
\begin{aligned}
y+\ol y &= 1 \\
\hfil y &= \af\, xy + \bf\, x + \ol\bf\, \ol y + \bf' + \ol\bf'\, x \ol y \\
\hfil \ol y &= \af\, x\ol y + \bf\, y + \ol\bf\, x  + \bf'\, x y + \ol\bf' \;,
\end{aligned}
\end{cases}
\end{equation}
and formula \eqref{eq:y} obtained by isolating $x$ from the last two equations is equivalent to their consistency. Now, for any $0<y<1$ that satisfies \eqref{eq:y} one obtains \eqref{eq:x} by summing the last two equations of system \eqref{eq:sys1} and taking into account that $y+\ol y=1$ by the first equation. This solution is ``probabilistically meaningful'' in the sense of \eqref{eq:pm}, i.e., $0<x<1$ provided that $0<y<1$, which (in combination with the already established uniqueness) yields the claim.
\end{proof}

\begin{rem}
The usual somewhat more cumbersome approach to unique solvability of the equations for passage probabilities on free groups and free products consists in using direct analytic considerations (see \textsc{Dynkin -- Malyutov} \cite[Theorem 2]{Dynkin-Malutov61} and \textsc{Ledrappier} \cite[Lemma 2.2]{Ledrappier01}) or fixed point theorems (see \textsc{Mairesse} \cite[Theorem~4.5 and Lemma 4.7]{Mairesse05}), whereas we obtain it as a consequence of uniqueness of stationary measures on the boundary.
\end{rem}

\begin{rem}
In general position the key equation \eqref{eq:y} is quadratic with respect to the variable $y$. However, the quadratic term may vanish for certain values of parameters~\eqref{eq:w}, for instance, in the ``symmetric case'' when $\bf=\ol\bf$ and $\bf'=\ol\bf'$.
\end{rem}

\thmref{thm:u} immediately implies

\begin{thm} \label{thm:t}
Under condition \eqref{eq:c}, the harmonic measure $\nu$ of random walk $(\G,\mu)$ belongs to the Denjoy class $\kab^\al$ (with $0<\al<1$) if and only if the weights \eqref{eq:w} of the step distribution $\mu$ satisfy relation
\begin{equation} \label{eq:g}
\begin{aligned}
&\Bigl[ \al - \ol\bf(1-\al) - \bf' \Bigr] \cdot \Bigl[ \af (1-\al) + \bf'\, \al + \ol\bf \Bigr] \\
&= \left[ \af\,\al + \ol\bf' (1- \al) + \bf \right] \cdot \left[ (1-\al) - \bf\,\al - \ol\bf' \right] \;.
\end{aligned}
\end{equation}
In particular, the harmonic measure $\nu$ belongs to the Minkowski class $\kab^{\frac12}$ if and only if
\begin{equation} \label{eq:g2}
\Bigl(1 - \ol\bf - 2\bf'\Bigr) \Bigl( \af + \bf' + 2\ol\bf \Bigr) =  \left( \af + \ol\bf' + 2\bf \right) \left( 1 - \bf - 2 \ol\bf' \right) \;.
\end{equation}
\end{thm}

\begin{rem}
Beyond the nearest neighbour case (see \secref{sec:nn} below and \remref{rem:nn}), the only result we are aware of with an explicit description of the harmonic measure of a \emph{non-degenerate} random walk on the modular group is due to \textsc{Letac -- Piccioni} \cite{Letac-Piccioni18} who considered the measure $\mu$ uniformly distributed on the 9-element set $\{e,a\}\cdot\left\{b,\ol b\right\}\cdot\{e,a\} \cup \{a\}$ and proved (by a rather involved probabilistic argument) that the harmonic measure in this case is, in our notation, $\ka^{\frac12,\frac12}=\frac12\left(\ka^{\frac12}+a\ka^{\frac12}\right)$. However, the directly verifiable fact that the convolution of each of the boundary measures $\ka^{\frac12}$ and $a\ka^{\frac12}$ with the measure $\frac12(\de_b+\de_{\ol b})$ on the group is a convex combination of $\ka^{\frac12}$ and $a\ka^{\frac12}$ immediately implies that the measure $\ka^{\frac12,\frac12}$ is $\mu$-stationary (and therefore is the harmonic measure by \prpref{prp:st}).
\end{rem}

\begin{rem} \label{rem:cf}
Extending our framework to \emph{degenerate} random walks on the modular group\,---\,or, more generally, on the group $\GL(2,\ZZ)$\,---\,naturally encompasses
the distributions of random continued fractions. In light of the discussion in \secref{sec:cf}, these distributions arise as harmonic measures of random walks on the semigroup $\Cc^*$ freely generated by the family of matrices $\Cc=\{C_n\}_{n=1}^\infty$ \eqref{eq:cn}. The singularity of these measures with respect to the \emph{Lebesgue measure}\footnotemark\ on the unit interval has been explored in a number of publications.

\footnotetext{\;One should keep in mind the distinction between two symbolic encodings of the unit interval: the binary one provided by the mediant tree (i.e., ultimately derived from the standard generating set of the modular group), and the infinite alphabet one provided by continued fractions. These two encodings give rise to distinct classes of Gibbs measures. Geometrically, this difference stems from the well-known fact that the distance on the modular group induced by the ambient Riemannian metric is \emph{not} quasi-isometric to a word distance. In particular, the Lebesgue measure class (which is conformal with respect to the induced distance) is Gibbs for the continued fractions encoding, but not for the mediant encoding. See \textsc{Kifer~-- Peres -- Weiss} \cite{Kifer-Peres-Weiss01}, \textsc{Jordan -- Sahlsten} \cite{Jordan-Sahlsten16}, and the references therein for a discussion of the thermodynamical formalism for the Gauss transformation.}

If the support of the step distribution $\mu$ is contained in the alphabet $\Cc$ (identified with the set of positive integers), then the associated harmonic measure $\nu$ is the distribution of random continued fractions with independent $\mu$-distributed digits, and therefore it is invariant with respect to the Gauss transformation. Since the Gauss transformation has an ergodic invariant measure (the \textsf{Gauss measure}) equivalent to the Lebesgue one, and the digits of continued fractions are not independent with respect to the Gauss measure, the measure $\nu$ is then singular with respect to the Lebesgue measure.\footnotemark\ The same argument applied to powers of the Gauss transformation shows that the harmonic measure is singular when $\mu$ is supported on the set $\Cc^n$ of words of the same length $n$ in the alphabet~$\Cc$ as well. However, we are not aware of any singularity results for other distributions on $\Cc^*$, and \emph{a priori} it is unclear whether the technique of Connell -- Muchnik (see \S\textbf{3} of the Introduction) could be applied in this setting to obtain absolutely continuous harmonic measures.

\footnotetext{\;\textsc{Chatterji} \cite{Chatterji66} and \textsc{Schweiger} \cite{Schweiger69} actually proved that any \emph{product measure} on the digits of continued fractions is singular with respect to the Lebesgue measure. As for the i.i.d.\ case (more generally, for a stationary $d$-Markov dependence with a fixed $d$), the Hausdorff dimension of the resulting measure on the unit interval is uniformly bounded away from 1, as shown by \textsc{Kifer -- Peres -- Weiss} \cite{Kifer-Peres-Weiss01}.}

If $\mu$ is concentrated on $\Cc$ and geometrically distributed with parameter $\frac12$, then its harmonic measure is precisely the Minkowski measure on the unit interval as pointed out already by \textsc{Chatterji} \cite{Chatterji66}. General Denjoy measures can also be obtained in this way by choosing appropriate step distributions on $\Cc^2$, see \textsc{Chassaing -- Letac -- Mora} \cite{Chassaing-Letac-Mora83}.
\end{rem}

\begin{rem}
Random walks on the modular group (to be more precise, on the group $\GL(2,\ZZ)$) also play a role in the study of random Fibonacci sequences, and the resulting harmonic measure is instrumental in finding their exponential growth rate. However, in this situation one is only interested in the step distributions $\mu$ charging just two matrices $\begin{pmatrix} 0 & \pm 1 \\1 & 1 \end{pmatrix}$. The harmonic measure is then the Minkowski measure if the weights of $\mu$ are equal \textsc{Viswanath} \cite{Viswanath00}, and an ``alternating'' version of the Denjoy measure in the general case \textsc{Janvresse -- Rittaud -- de la Rue} \cite{Janvresse-Rittaud-delaRue08}.
\end{rem}

\subsection{The nearest neighbour case} \label{sec:nn}

If the measure $\mu$ only charges the set of generators $\Ac=\left\{a,b,\ol b\right\}$, whereas the weights $\mu(ba) = \bf'$ and $\mu\left(\ol ba\right) = \ol\bf'$ are both zero, i.e., if the random walk $(\G,\mu)$ is a \emph{nearest neighbour} one, formulas from \secref{sec:pass} significantly simplify. In order to use their symmetry to the full extent it is convenient (keeping the notation of \thmref{thm:u}) to introduce a new variable
$$
z = y - \ol y = 2y - 1
$$
and pass from parameters $\bf,\ol\bf$ (with $\bf+\ol\bf=1-\af$, because $\bf'=\ol\bf'=0$) in the description of the measure~$\mu$ to new parameters $\af$ and $\deu = \bf - \ol\bf$, so that in this situation non-degeneracy condition~\eqref{eq:c} becomes
\begin{equation} \label{eq:afd}
0<\af<1\;, \qquad |\deu|\le 1- \af \;.
\end{equation}

Then equations \eqref{eq:y} and \eqref{eq:x} from \thmref{thm:u} take, respectively, the form
\begin{equation} \label{eq:ee2}
\af\deu\, z^2 + \bigl( 4 - (\af+1)^2 + \deu^2 \bigr) z - \af\deu = 0 \;,
\end{equation}
and
\begin{equation} \label{eq:xd}
x = \frac12 (1+\af-\deu z) \;.
\end{equation}
If $\deu=0$, i.e., if $\bf=\ol\bf$, then the trivial unique solution of equation \eqref{eq:ee2} is $z=0$, whereas for $\deu\neq 0$ we have
\begin{equation} \label{eq:eq}
z^2 + 2 D z -1 = 0 \quad\text{with}\quad D = \frac{4 - (\af+1)^2 + \deu^2}{2\af\deu} \;,
\end{equation}
and the unique meaningful solution of this equation (i.e., the one with $|z|<1$) is
\begin{equation} \label{eq:z}
z = (\sgn\deu) \left(\sqrt{D^2+1} - |D|\right)
\end{equation}
(so that the signs of $\deu, D$, and $z$ all coincide).

In view of \thmref{thm:ka}, formulas \eqref{eq:xd} and \eqref{eq:z} provide then an explicit description of the harmonic measures of the nearest neighbour random walks on group $\G$.

\begin{prp} \label{prp:nn}
Let $\mu$ be a probability measure on the set of generators $\Ac=\left\{a,b,\ol b\right\}$ of group $\G$ with the weights
\begin{equation} \label{eq:step}
\mu(a)=\af\;, \qquad \mu(b)=\bf=\frac12 (1-\af+\deu) \;, \qquad \mu\left(\ol b\right)=\ol\bf=\frac12 (1-\af-\deu) \;,
\end{equation}
where $\af$ and $\deu$ are subject to condition \eqref{eq:afd}. Then the parameters of the associated harmonic measure $\nu=\ka^{\al,p}$ are
$$
\al = \frac{1+z}2 \;, \qquad p = \frac{1+\af-\deu z}{3+\af-\deu z} \;,
$$
where $z=0$ when $\deu=0$, and otherwise $z$ is the solution \eqref{eq:z} of equation \eqref{eq:eq}.
\end{prp}

\begin{cor} \label{cor:nn1}
The harmonic measures of two nearest neighbour random walks on group~$\G$ are equivalent if and only if the respective parameters $\af,\deu$ \eqref{eq:afd} of their step distributions~\eqref{eq:step} belong to the same level set of the function
\begin{equation} \label{eq:ad}
\Phi(\af,\deu) = \frac{\af\deu}{4 - (\af+1)^2 + \deu^2} \;.
\end{equation}
\end{cor}

\begin{cor} \label{cor:nn2}
The harmonic measure of a nearest neighbour random walk on group~$\G$ belongs to the Minkowski class $\kab^{\frac12}$ if and only if its step distribution $\mu$ is symmetric, i.e., $\mu(b)=\mu\left(\ol b\right)$.
\end{cor}

\begin{rem} \label{rem:nn}
The transition probabilities of a symmetric nearest neighbour random walk on $\G$ are invariant with respect to the action of the group $\Rc$ of rotations of the unlabelled Cayley graph, whence the harmonic measure is also $\Rc$-invariant, and therefore it is necessarily a convex combination $\ka^{\frac12,p}$ of the measures $\ka^{\frac12}$ and $a\ka^{\frac12}$, see \remref{rem:rot}. The coefficient $p$ can be easily found from the $\mu$-stationarity condition on the harmonic measure (or by a direct probabilistic consideration of the associated birth-and-death chain, cf.\ the case of the simple random walk on a finitely generated free group). This exercise was carried out in 1978 by the second author, then a third-year undergraduate student, at the request of his teacher Anatoly Vershik.

For general nearest neighbour random walks on free products the Green kernel was independently described by \textsc{Cartwright -- Soardi} \cite{Cartwright-Soardi86} and by \textsc{Woess} \cite{Woess86a}. The resulting harmonic measure was then explicitly exhibited by \textsc{Woess} \cite[Theorem~4]{Woess86} in terms of the Green kernel. However, the dependence of the Green kernel (or, equivalently, of the passage probabilities of the random walk) on the step distribution was first investigated much later by \textsc{Mairesse -- Matheus} \cite{Mairesse-Matheus07}. In particular, a description of the dependence of the base distribution of the harmonic measure (considered as a multiplicative Markov measure, see \remref{rem:tmc}) on the weights of the step distribution was given in \cite[Section 4.2]{Mairesse-Matheus07}. These formulas being quite bulky, we did not even attempt to verify their equivalence to \prpref{prp:nn} analytically; however, they do agree with ours numerically.
\end{rem}

\subsection{Another particular case} \label{sec:ano}

The step distributions of the nearest neighbour random walks considered in \secref{sec:nn} belong to the 2-dimensional convex subset of the simplex of probability measures on set $\Ac=\left\{a,b,\ol b,ba,\ol b a\right\}$ which significantly simplified our analysis. For the purpose of constructing our examples we will need yet another 2-dimensional convex set of step distributions described, in terms of notation \eqref{eq:w}, by requiring that
$\bf=\bf'$ and $\ol\bf'=0$, so that $\af=1-2\bf-\ol\bf$, and therefore
\begin{equation} \label{eq:ano}
\mu = \left(1-2\bf-\ol\bf\right)\de_a + \bf \left( \de_b + \de_{ba} \right) + \ol\bf\, \de_{\ol b} \;.
\end{equation}
In view of condition \eqref{eq:c}, we have to exclude two vertices $(0,0)$ and $(0,1)$ from the triangle $\left\{\left(\bf,\ol\bf\right): \bf,\ol\bf\ge 0, 2\bf+\ol\bf\le 1\right\}$ coloured in grey in \figref{fig:hyp}. By using formula \eqref{eq:g2} from \thmref{thm:t} we can then conclude that the corresponding harmonic measure belongs to the Minkowski class $\kab^{\frac12}$ if and only if
$$
\left(1-2\bf-\ol\bf\right)\left(1-\bf+\ol\bf\right)=(1-\bf)\left(1-\ol\bf\right) \;,
$$
or, after a simplification,
\begin{equation} \label{eq:anocon}
2\bf^2 - 2\bf\ol\bf - \ol\bf^2 -2\bf + \ol\bf =0 \;,
\end{equation}
which describes a hyperbola in the $\left(\bf,\ol\bf\right)$ plane, see \figref{fig:hyp}.

\begin{figure}[h]
\begin{center}
\includegraphics[scale=.6]{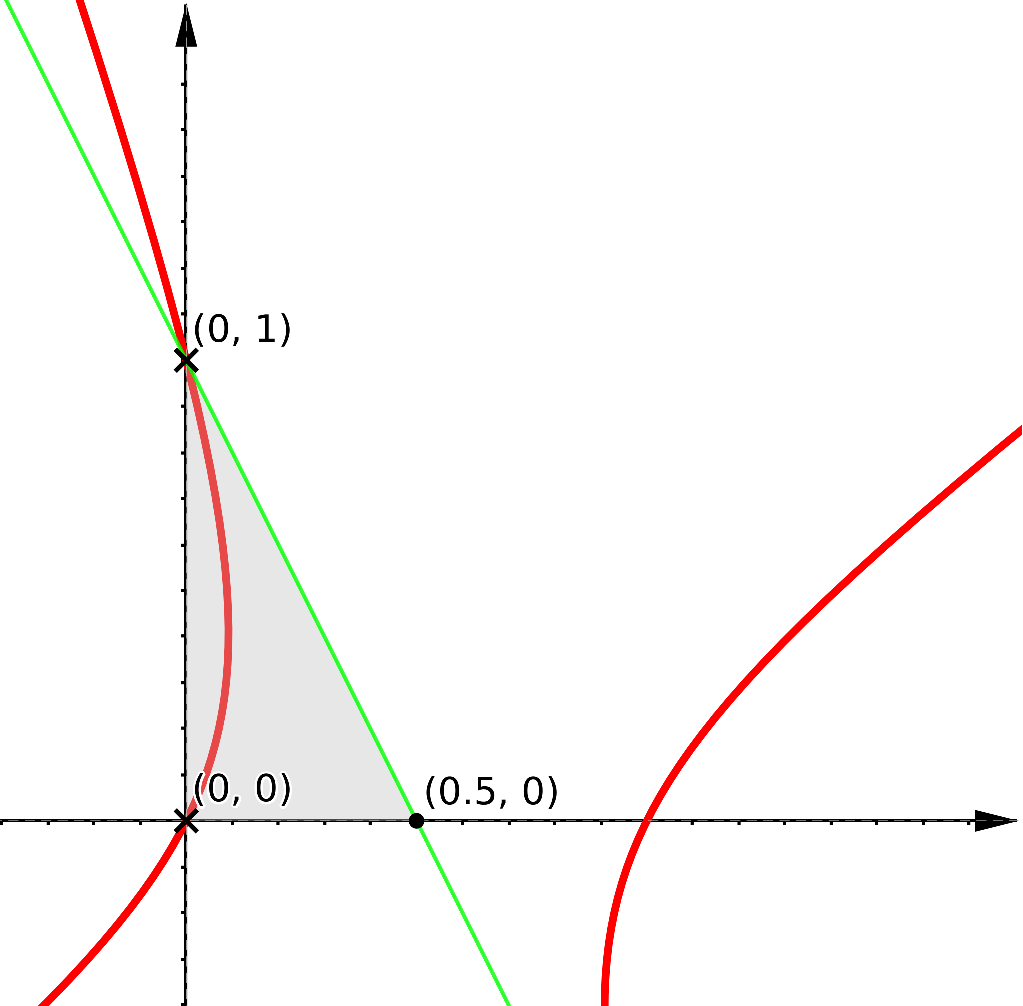}
\end{center}
\caption{Hyperbola \eqref{eq:anocon} on the $\left(\bf,\ol\bf\right)$ plane that represents those of step distributions~\eqref{eq:ano} whose harmonic measure belongs to the Minkowski class.}
\label{fig:hyp}
\end{figure}

\subsection{Singularity examples} \label{sec:sex}

Having obtained explicit descriptions of the harmonic measure for a family of random walks on group $\G$, we can now return to the singularity questions formulated in the Introduction. Our first example can be given already within the class of the nearest neighbour random walks by using \prpref{prp:nn} and \corref{cor:nn1}.

\begin{tme} \label{tme:ex0}
There exist two non-degenerate probability measures $\mu_1,\mu_2$ on the generating set $\Ac=\left\{a,b,\ol b\right\}$ of group $\G$ such that the corresponding harmonic measures~$\nu_1$ and $\nu_2$ are equivalent, whereas for any non-trivial convex combination $\mu=t\mu_1+(1-t)\mu_2$ with $0<t<1$ its harmonic measure $\nu$ is singular to $\nu_1$ and $\nu_2$.
\end{tme}

\begin{proof}
\corref{cor:nn1} explicitly spells out a necessary and sufficient condition for equivalence of the harmonic measures of two nearest neighbour random walks on group $\G$, and it is obvious that the level sets $\Phi^{-1}(c)$ of function $\Phi$ \eqref{eq:ad} are not convex for $c\neq 0$.
\end{proof}

However, the class of nearest neighbour random walks is not big enough to produce any examples concerning the harmonic measure of the convolution of two step distributions. Moreover, in view of \corref{cor:nn2}, the harmonic measures $\nu_1$ and $\nu_2$ in \tmeref{tme:ex0} cannot be taken from the Minkowsky ($\equiv$ Hausdorff) measure class on the boundary. In order to address these drawbacks we have to use the bigger class of random walks described by condition \eqref{eq:c} and investigated in \thmref{thm:ka} and \thmref{thm:t}.

\begin{tme} \label{tme:ex1}
There exist two non-degenerate probability measures $\mu_1,\mu_2$ on the subset $\Sc=\left\{a,b,\ol b,ba,\ol ba \right\}$ of group $\G$ such that the corresponding harmonic measures~$\nu_1$ and $\nu_2$ both belong to the Minkowsky measure class on the boundary, whereas for any non-trivial convex combination $\mu=t\mu_1+(1-t)\mu_2$ with $0<t<1$ its harmonic measure~$\nu$ is singular to this measure class.
\end{tme}

\begin{proof}
It is pretty straightforward that the subset of the simplex of probability measures on~$\Sc$ determined by condition \eqref{eq:g2} from \thmref{thm:t} is not convex. For being more explicit, we can pass to the 2-dimensional convex subset of this simplex described in \secref{sec:ano} (the grey triangle in \figref{fig:hyp}). Then the step distributions corresponding to any two distinct points on the hyperbola branch between the points $(0,0)$ and $(0,1)$ provide a desired example.
\end{proof}

Although the class of step distributions described by condition \eqref{eq:c} is not closed with respect to convolutions, we can bypass this obstacle by using the following observation. If one passes from a step distribution $\mu$ on $\G$ with the harmonic measure $\nu$ to its conjugate~$g\mu g^{-1}$ by a group element $g\in\G$, then the translate $g\nu$ is obviously $g\mu g^{-1}$-stationary, and therefore the harmonic measure of the step distribution $g\mu g^{-1}$ is $g\nu$. In particular, quasi-invariance of $\nu$ implies that the harmonic measures of the step distributions~$\mu$ and~$g\mu g^{-1}$ are equivalent. On the other hand, if $g$ is the order 2 generator $a$ of group $\G$, then the convolution of $\mu$ and the conjugate $a\mu a$ coincides with the convolution square~$(\mu a)^{*2}$ of the translate $\mu a$. The harmonic measures of the step distributions~$\mu a$ and~$(\mu a)^{*2}$ being the same, we arrive at the following conclusion: if the harmonic measures of $\mu$ and $\mu a$ are singular, then the harmonic measure of the convolution of the step distributions $\mu$ and $a\mu a$ is singular with respect to the common measure class of the harmonic measures of $\mu$ and $a\mu a$. In view of this observation, we can now exhibit our last example.

\begin{tme} \label{tme:ex2}
There exist two non-degenerate probability measures $\mu_1,\mu_2$ on the subset $\Sc=\left\{a,b,\ol b,ba,\ol ba \right\}$ of group $\G$ such that the corresponding harmonic measures $\nu_1$ and $\nu_2$ both belong to the Minkowsky measure class on the boundary, whereas the harmonic measure~$\nu$ determined by the convolution $\mu_1*\mu_2$ is singular with respect to this measure class.
\end{tme}

\begin{proof}
We take for $\mu_1$ a measure \eqref{eq:ano} subject to condition \eqref{eq:anocon} described in \secref{sec:ano} and already used in the proof of \tmeref{tme:ex1} (it corresponds to a pair of parameters $\left(\bf,\ol\bf\right)$ on the hyperbola branch between the points $(0,0)$ and $(0,1)$ in \figref{fig:hyp}), and put $\mu_2=a\mu_1 a$. Then the harmonic measures of both step distributions $\mu_1$ and $\mu_2$ belong to the Minkowski measure class. Now we have to make sure that the harmonic measure~$\nu$ of the step distribution $\mu=\mu_1*\mu_2$ (i.e., as explained above, of the step distribution $\mu_1 a$) is not in the Minkowski class. After removing the weight $\left( 1-2\bf-\ol\bf \right)$ of the measure $\mu_1 a$ at the group identity and subsequent renormalization, we obtain the step distribution
$$
\mu' = \frac{\bf}{2\bf+\ol\bf} \left( \de_b + \de_{ba} \right) + \frac{\ol\bf}{2\bf+\ol\bf} \de_{\ol b a}
$$
with the same harmonic measure $\nu$. On the other hand, by \thmref{thm:t} applied to the measure $\mu'$, we see that its harmonic measure belongs to the Minkowski class if and only if
$$
2\bf^2 - 2\bf\ol\bf - \ol\bf^2 =0 \;,
$$
which is incompatible with equation \eqref{eq:anocon}.
\end{proof}


\providecommand{\MR}{}
\providecommand{\MRhref}[2]{}
\providecommand{\href}[2]{#2}

\end{document}